\documentclass[11pt, reqno]{amsart}

\usepackage[top=3.75cm, bottom=3cm, left=3cm, right=3cm]{geometry}
\usepackage{amsmath}
\usepackage{amsfonts}
\usepackage{amssymb}
\usepackage{amsthm}
\usepackage[usenames]{xcolor}
\usepackage{graphicx}
\usepackage[all]{xy}
\usepackage{url}
\usepackage[abs]{overpic}
\usepackage[hyperfootnotes=false]{hyperref}
\usepackage{verbatim}
\usepackage{todonotes}
\usepackage{tikz}
\usepackage{tikz-cd}
\usetikzlibrary{arrows,calc,matrix,patterns,shadings,backgrounds,fit}
\usepackage{lipsum}
\usepackage{mathtools}
\usepackage{stmaryrd}
\usepackage{environ}

\makeatletter
\newdimen\imgresize@width
\newdimen\imgresize@threshold
\makeatother

\theoremstyle{definition}
\newtheorem{defn}{Definition}[section]
\newtheorem{ex}[defn]{Example}
\newtheorem{rmk}[defn]{Remark}
\newtheorem{cond}[defn]{Condition}

\newtheorem{question}[defn]{Question}
\theoremstyle{plain}
\newtheorem{thm}[defn]{Theorem}
\newtheorem{lem}[defn]{Lemma}
\newtheorem{prop}[defn]{Proposition}
\newtheorem{cor}[defn]{Corollary}

\newtheorem{conj}[defn]{Conjecture}
\newtheorem{PropDef}[defn]{Proposition and Definition}
\newtheorem{sublem}[defn]{Sublemma}
\numberwithin{equation}{section}
\def\C{\ensuremath{\mathbb{C}}}

\def\L{\ensuremath{\mathbb{L}}}
\def\N{\ensuremath{\mathbb{N}}}
\def\P{\ensuremath{\mathbb{P}}}

\def\R{\ensuremath{\mathbb{R}}}
\def\Z{\ensuremath{\mathbb{Z}}}

\def\AA{\ensuremath{\mathcal A}}
\def\BB{\ensuremath{\mathcal B}}
\def\CC{\ensuremath{\mathcal C}}
\def\DD{\ensuremath{\mathcal D}}
\def\EE{\ensuremath{\mathcal E}}
\def\FF{\ensuremath{\mathcal F}}
\def\GG{\ensuremath{\mathcal G}}
\def\HH{\ensuremath{\mathcal H}}
\def\II{\ensuremath{\mathcal I}}

\def\KK{\ensuremath{\mathcal K}}
\def\LL{\ensuremath{\mathcal L}}
\def\MM{\ensuremath{\mathcal M}}

\def\OO{\ensuremath{\mathcal O}}

\def\TT{\ensuremath{\mathcal T}}

\def\XX{\ensuremath{\mathcal X}}
\def\YY{\ensuremath{\mathcal Y}}
\def\ZZ{\ensuremath{\mathcal Z}}

\def\Aut{\mathop{\mathrm{Aut}}\nolimits}

\def\ch{\mathop{\mathrm{ch}}\nolimits}

\def\Coh{\mathop{\mathrm{Coh}}\nolimits}

\def\Db{\mathop{\mathrm{D}^{\mathrm{b}}}\nolimits}

\def\dim{\mathop{\mathrm{dim}}\nolimits}

\def\ext{\mathop{\mathrm{ext}}\nolimits}
\def\Ext{\mathop{\mathrm{Ext}}\nolimits}
\def\hom{\mathop{\mathrm{hom}}\nolimits}
\def\Hom{\mathop{\mathrm{Hom}}\nolimits}

\def\RlHom{\mathop{\mathbf{R}\mathcal Hom}\nolimits}

\def\id{\mathop{\mathrm{id}}\nolimits}

\def\Ker{\mathop{\mathrm{Ker}}\nolimits}

\def\rk{\mathop{\mathrm{rk}}}

\def\Spec{\mathop{\mathrm{Spec}}}

\def\supp{\mathop{\mathrm{supp}}}
\def\Sym{\mathop{\mathrm{Sym}}}

\def\Stab{\mathop{\mathrm{Stab}}\nolimits}

\def\Ku{\mathop{\mathrm{Ku}}\nolimits}
\def\Forg{\mathop{\mathrm{Forg}}\nolimits}
\def\dlt{\Delta_{\mathcal B_0}}

\def\chbl{\ch^{-1}_{\BB_0,\leq 2}}
\def\zab{Z_{\alpha,\beta}}

\def\Ms{M^{\text{sing}}}
\def\Zs{Z^{\text{sing}}}
\newcommand{\tM}{\widetilde{M}}

\newcommand{\tJ}{\widetilde{J}}

\def\Def{\text{Def}}
\def\kn{\mathrm{K}_{\mathrm{num}}}

\author{Chunyi Li}
\address{C.\ L.: Mathematics Institute (WMI), University of Warwick, Coventry, CV4 7AL, United Kingdom.}
\email{C.Li.25@warwick.ac.uk}
\urladdr{https://sites.google.com/site/chunyili0401/}
\author{Laura Pertusi}
\address{L.\ P.: Dipartimento di Matematica F.\ Enriques, Universit\`a degli studi di Milano, Via Cesare Saldini 50, 20133 Milano, Italy.}
\email{laura.pertusi@unimi.it}
\urladdr{http://www.mat.unimi.it/users/pertusi/}
\author{Xiaolei Zhao}
\address{X.\ Z.: Department of Mathematics, University of California, Santa Barbara, South Hall 6705, Santa Barbara, CA 93106, USA.}
\email{xlzhao@ucsb.edu}
\urladdr{https://sites.google.com/site/xiaoleizhaoswebsite/}

\begin{document}
\title[Elliptic quintics on cubic fourfolds]{Elliptic quintics on cubic fourfolds, O'Grady 10, and Lagrangian fibrations}
\maketitle 

\begin{abstract}
For a smooth cubic fourfold $Y$, we study the moduli space $M$ of semistable objects of Mukai vector $2\lambda_1+2\lambda_2$ in the Kuznetsov component of $Y$. We show that with a certain choice of stability conditions, $M$ admits a symplectic resolution $\tM$, which is a smooth projective hyperk\"ahler manifold, deformation equivalent to the $10$-dimensional examples constructed by O'Grady. As applications, we show that a birational model of $\tM$ provides a hyperk\"ahler compactification of the twisted family of intermediate Jacobians associated to $Y$. This generalizes the previous result of Voisin \cite{Voisin:twisted} in the very general case. We also prove that $\tM$ is the MRC quotient of the main component of the Hilbert scheme of elliptic quintic curves in $Y$, confirming a conjecture of Castravet.
\end{abstract}

\section{Introduction}

Moduli spaces of stable sheaves on a K3 surface provide the major examples of projective hyperk\"ahler manifolds. These examples are deformation equivalent to Hilbert schemes of points on a K3 surface, by the seminal work of Mukai \cite{Mukai:BundlesK3} and the contribution of many other authors, including Beauville \cite{Beauville:remarksonc1zero}, O’Grady \cite{OG:weight2}, Yoshioka \cite{Yoshioka:Irreducibility, Yoshioka:Abelian}. In \cite{OGrady}, O'Grady considered the case when the moduli space contains also strictly semistable sheaves. In particular, he constructed a symplectic resolution of the singular moduli space of semistable torsion-free sheaves on a K3 surface with rank $2$, trivial first Chern class and second Chern class equal to $4$. This construction provides a new example of a hyperk\"ahler manifold of dimension $10$, not deformation equivalent to the previous construction. O'Grady's result was generalized by Lehn and Sorger in \cite{LS} to moduli spaces of semistable sheaves on a K3 surface having Mukai vector of the form $v=2v_0$ with $v_0^2=2$. In addition, they showed that the symplectic resolution of the moduli space can be obtained by blowing up the singular locus with the reduced scheme structure.

In this paper we investigate the analogous situation of O'Grady's example, in the case of moduli spaces of semistable complexes in the noncommutative K3 surface associated to a smooth cubic fourfold. By \cite{Kuz:fourfold}, the bounded derived category of a cubic fourfold $Y$ has a semiorthogonal decomposition of the form
$$\Db(Y)=\langle\Ku(Y),\OO_Y,\OO_Y(H),\OO_Y(2H)\rangle,$$
where $H \subset Y$ is a hyperplane section and $\Ku(Y)$ is a triangulated subcategory of K3 type, in the sense that it has the same Serre functor and Hochschild homology as the derived category of a K3 surface \cite[Corollary 4.3]{Kuz:V14}, \cite[Proposition 4.1]{Kuz:rat}. We call this category $\Ku(Y)$ the \emph{Kuznetsov component} of $Y$. One reason to study $\Ku(Y)$ is related to the birational geometry of $Y$. For instance, there is a folklore conjecture  \cite[Conjecture 1.1]{Kuz:fourfold} that $Y$ is rational if and only if $\Ku(Y)$ is equivalent to the derived category of a K3 surface. %and verified this conjecture for several known examples of rational cubic fourfolds.

Another interest in studying $\Ku(Y)$ is to generalize Mukai's construct to this noncommutative K3 surface. Bayer, Lahoz, Macr\`i and Stellari construct Bridgeland stability conditions on $\Ku(Y)$ in \cite{BLMS:kuzcomponent} (see Section \ref{section_preliminaries} for a review of the construction). We denote by $\Stab^\dag(\Ku(Y))$ the connected component of the stability manifold containing these stability conditions. In a second paper \cite{BLM+}, joint also with Nuer and Perry, they develop the deep theory of families of stability conditions, which allows studying the properties of moduli spaces of stable objects in $\Ku(Y)$ by deforming to cubic fourfolds whose Kuznetsov components are equivalent to the derived category of a K3 surface. As a consequence, they produced infinite series of unirational, locally complete families of smooth polarized hyperk\"ahler manifolds, deformation equivalent to Hilbert schemes of points on a K3 surface. These hyperk\"ahler manifolds are given as moduli spaces of stable objects in $\Ku(Y)$ of primitive Mukai vector. It is worth to point out that the hyperk\"ahler manifolds constructed from some Hilbert schemes of rational curves of low degree in $Y$ can be interpreted as moduli spaces of stable objects in $\Ku(Y)$. Indeed, we gave in \cite{LPZ1} a description of the Fano variety of lines in $Y$ \cite{Beauville-cubic4fold} and, when $Y$ does not contain a plane, of the hyperk\"ahler $8$-fold constructed in \cite{LLSvS} using twisted cubic curves in $Y$, as moduli spaces of stable objects in $\Ku(Y)$ with primitive Mukai vector.

In analogy to the case of K3 surfaces, the Mukai lattice of $\Ku(Y)$ has been defined in \cite{AddingtonThomas:CubicFourfolds} and carries a weight two Hodge structure induced from that on the cohomology of $Y$. We denote by $H^*_{\mathrm{alg}}(\Ku(Y),\Z)$ the sublattice of integral $(1,1)$ classes in the Mukai lattice of $\Ku(Y)$ (see Section \ref{subsec_algMukai}).

Consider now a vector $v=2v_0 \in H^*_{\mathrm{alg}}(\Ku(Y),\Z)$ such that $v_0$ is primitive with $v_0^2=2$. Let $\tau$ be a stability condition in $\Stab^\dag(\Ku(Y))$ which is generic with respect to $v$, in other words, the strictly $\tau$-semistable objects with Mukai vector $v$ are (S-equivalent to) direct sums of $\tau$-stable objects with Mukai vector $v_0$. Let $M$ be the moduli space of $\tau$-semistable objects with Mukai vector $v$. The first result of this paper is the following.
\begin{thm}[Theorem \ref{thm_OG10}]
\label{thm_OG10intro}
The moduli space $M$ has a symplectic  resolution $\tM$, which is a $10$-dimensional smooth projective hyperk\"ahler manifold, deformation equivalent to the O'Grady's example constructed in \cite{OGrady}. 
\end{thm}

In the second part we explain two main applications, which make a connection between the derived categorical viewpoint of Theorem \ref{thm_OG10intro} and the classical construction of hyperk\"ahler manifolds from $Y$. Recall that by \cite{AddingtonThomas:CubicFourfolds}, the algebraic Mukai lattice of $\Ku(Y)$ contains two classes $\lambda_1$ and $\lambda_2$ spanning an $A_2$-lattice. Motivated by classical geometric constructions (as it will be clear later), we consider the case $v_0=\lambda_1+\lambda_2$, $v=2v_0$ and we analyze the objects in $M:=M_{\sigma}(v)$ where $\sigma$ is a stability condition as constructed in \cite{BLMS:kuzcomponent}. It is not difficult to see that by \cite{LPZ1} the strictly semistable locus of $M$ is identified with the symmetric square of the Fano variety of lines in $Y$, up to a perturbation of the stability condition (see Remark \ref{rem:ssstable}). On the other hand, stable objects are harder to describe. If $X$ is a smooth hyperplane section of $Y$, in other words, $X$ is a smooth cubic threefold, then the moduli space $M_{\mathrm{inst}}$ parametrizing rank $2$ instanton sheaves on $X$ have been described by \cite{Druel:Instanton}. In particular, stable sheaves in $M_{\mathrm{inst}}$ belong to one of the following classes: rank $2$ stable vector bundles constructed from non-degenerate elliptic quintics in $X$,  rank $2$ stable torsion free sheaves associated to smooth conics in $X$. Moreover, the strictly semistable objects in $M_{\mathrm{inst}}$ are direct sums of two ideal sheaves of lines in $X$ (see Section \ref{subsec_instanton} for a review). By \cite{Beauville:Cubics, Druel:Instanton} the moduli space  $M_{\mathrm{inst}}$ is birational to the translate $J^2(X)$ of the intermediate Jacobian, which parameterizes $1$-cycles of degree $2$ on $X$.

Denote by $\sigma$ a stability condition constructed in \cite{BLMS:kuzcomponent}. A key result for our applications is the following theorem, which provides a description of an open subset of the stable locus of $M:=M_{\sigma}(2\lambda_1+2\lambda_2)$.

\begin{thm}[Theorem \ref{thm:EgammaECinkustab}]
\label{thm:EgammaECinkustabintro}
Let $X$ be a smooth hyperplane section of $Y$. Then the projection in $\Ku(Y)$ of the stable rank $2$ instanton sheaves associated to non degenerate elliptic quintic curves and smooth conics in $X$ are $\sigma$-stable objects with Mukai vector $2\lambda_1+2\lambda_2$.
\end{thm}

We apply Theorem \ref{thm:EgammaECinkustabintro} to show that, up to a perturbation of the stability condition $\sigma$ in $\Stab^\dag(\Ku(Y))$ (see Section \ref{sec:Lagrangian}), the sympletic resolution $\tM$, given by Theorem \ref{thm_OG10intro} has a deep connection to a classical construction of Jacobian fibration associated to $Y$. Consider the $(\P^{5})^\vee$-family of cubic threefolds obtained as hyperplane sections of $Y$ and let $\P_0$ be its smooth locus. Consider the twisted family of intermediate Jacobians $p: J \to \P_0$, whose fibers are the twisted intermediate Jacobians of the smooth cubic threefolds parametrized by $\P_0$. It is known that there exists a holomorphic symplectic form on $J$ by \cite{DonagiMarkman}. However, it remained a long standing question whether $J$ can be compactified to a hyperk\"ahler manifold $\bar{J}$ and a Lagrangian fibration $\bar{J} \to (\P^{5})^\vee$ extending $p$. This has been recently proved for very general cubic fourfolds in the beautiful works \cite{LSV} for the untwisted family and \cite{Voisin:twisted} by Voisin for $J$. We mention that  in the recent preprint \cite{Sacca:birgeomJac}, Saccà extended the result for the untwisted family in \cite{LSV} to all smooth cubic fourfolds. The same argument applies to the twisted family and extends Voisin's result to all smooth cubic fourfolds (see \cite[Remark 1.10]{Sacca:birgeomJac}).

Our main result is the following modular construction of a hyperk\"ahler compactification of $J$ for every cubic fourfold $Y$, obtained combining Theorems \ref{thm_OG10intro}, \ref{thm:EgammaECinkustabintro} and some techniques in birational geometry of hyperk\"ahler varieties.

\begin{thm} [Propositions \ref{birational_model}, \ref{prop_compactification}]
\label{thm_compacttwistedJac}
There exists a hyperk\"ahler manifold $N$ birational to $\tM$, which admits a Lagrangian fibration structure compactifying the twisted intermediate Jacobian family $J \to \P_0$.
\end{thm}

It is worth to note that $N$ and $\tM$ are birational, but not isomorphic if $Y$ is very general. In Example \ref{ex_flopconics} we describe an explicit flop between them, involving the locus of stable objects in $\Ku(Y)$ coming from the projection of instanton sheaves associated to smooth conics in $Y$. In Remark \ref{rmk:compare_with_Voisin}, we explain how $N$ is related to the compactification constructed by Voisin \cite{Voisin:twisted}.

The next application arises from the following conjecture of Castravet. Note that the original conjecture involves rational quartics, but it can be equivalently stated for elliptic quintics by residuality (see Remark \ref{rmk_conjCastravet}).

\begin{conj}[{\cite[Page 416]{deJong_starr}}]
\label{conj_Castravet}
Let $\CC$ be the connected component of the Hilbert scheme $\emph{Hilb}^{5m}(Y)$ containing elliptic quintics in $Y$. Then the maximally rationally connected quotient of $\CC$ is birationally equivalent to the twisted intermediate Jacobian of $Y$.
\end{conj}

\noindent Using Theorems \ref{thm_OG10intro}, \ref{thm:EgammaECinkustabintro} and \ref{thm_compacttwistedJac} we are able to prove Conjecture \ref{conj_Castravet}.

\begin{prop}[Propositions \ref{prop_objsinmodulispace}, \ref{prop:quartics}]
\label{prop_objsinmodulispaceintro}
The projection functor (see Defintion \ref{def:EGammaEC}) induces a rational map $\CC \dashrightarrow M$ which is the maximally rationally connected fibration of $\CC$. The maximally rationally connected quotient of $\CC$ is birational to the the twisted family $J$ of intermediate Jacobians of $Y$.
\end{prop}

\noindent \textbf{Plan of the paper.} In Section \ref{section_preliminaries} we review some definitions and results about (weak) stability conditions on triangulated categories and semiorthogonal decompositions. In particular, we recall the construction of stability conditions on the Kuznetsov component $\Ku(Y)$ of a cubic fourfold $Y$ as in \cite{BLMS:kuzcomponent}. 

Section \ref{sec_sympleresol} is devoted to the proof of Theorem \ref{thm_OG10intro}. For an element $v_0$ with square $2$ in the algebraic Mukai lattice of $\Ku(Y)$, consider a stability condition $\tau$ on $\Ku(Y)$ which is $2v_0$-generic. We show that the blow-up $\tM$ of the singular locus of the moduli space $M:=M_\tau(2v_0)$ with the reduced scheme structure is a symplectic resolution, by describing the local structure of $M$ at the worst singularity, as done in \cite{LS} for singular moduli spaces on K3 surfaces. 

In Section \ref{sec:stabobjinM} we compute the projection in the Kuznetsov component of some objects related to elliptic quintics and smooth conic curves in a cubic fourfold. We explain their relation with stable instanton sheaves on smooth hyperplane sections of $Y$, which were previously studied in \cite{Druel:Instanton}. 

Section \ref{sec:stabE} deals with the proof of Theorem \ref{thm:EgammaECinkustabintro}. We show that the objects in the Kuznetsov component, constructed out of elliptic quintics and conics in $Y$, are $\sigma$-stable, where $\sigma$ is any stability condition as constructed in \cite{BLMS:kuzcomponent}. In particular, they describe an open subset of the moduli space $M_\sigma(2\lambda_1+2\lambda_2)$. Recall that $\sigma$ is induced on $\Ku(Y)$ from the restriction of (a tilt of) a weak stability condition $\sigma_{\alpha, -1}$ on the bounded derived category of coherent $\BB_0$-modules on $\P^3$, depending on a real parameter $\alpha>0$. Here $\BB_0$ is the even part of the sheaf of Clifford algebras associated to the conic fibration on $\P^3$ obtained by blowing-up a line in $Y$ (see Section \ref{subsec_kuzcompoinP3}, Proposition and Definition \ref{prop:stabonku}). In Section \ref{subsec_expressionGamma} we compute the expression of our objects as complexes of $\BB_0$-modules on $\P^3$. Then in Sections \ref{subsec_tiltstablGamma} and \ref{subsec_tiltstabC} we show they are $\sigma_{\alpha,-1}$-stable for $\alpha$ sufficiently large. Finally in Section \ref{sec:nowall} we show they are $\sigma_{\alpha,-1}$-stable for every $\alpha$, proving there are no walls for stability. 

In Section \ref{sec:Lagrangian} we prove Theorem \ref{thm_compacttwistedJac}. Fix a stability condition $\sigma_0$ which is generic with respect to $2\lambda_1+2\lambda_2$ and with the same stable objects as $\sigma$. Applying Theorems \ref{thm_OG10intro} and \ref{thm:EgammaECinkustabintro} to the moduli space $M:=M_{\sigma_0}(2\lambda_1+2\lambda_2)$, we consider the open subvariety $M_0$ of the symplectic resolution $\tM$ of $M$ consisting of stable objects associated to elliptic quintics in $Y$, with support on smooth hyperplane sections of $Y$. We define a line bundle $\LL$ on $\tM$ inducing a rational map from $\tM$ to $\P^{5\vee}$, which is defined on $M_0$ by sending the object to its support. Using some results in birational geometry of hyperk\"ahler varieties, we show that there is a birational model $N$ of $\tM$ and a semiample line bundle $\LL'$ on $N$, such that a multiple of $\LL'$ induces a Lagrangian fibration structure on $N$ which is a compactification of the twisted intermediate Jacobian $J$ over $\P_0$.

We conclude with Section \ref{sec_conjCastravet} where Conjecture \ref{conj_Castravet} is proved as a consequence of Theorems \ref{thm_OG10intro}, \ref{thm:EgammaECinkustabintro} and \ref{thm_compacttwistedJac}.\\

\noindent \textbf{Acknowledgements.}
This paper benefited from many useful conversations with Arend Bayer, Jack Hall, Daniel Huybrechts, Mart\'i Lahoz, Manfred Lehn, Emanuele Macrì, Alex Perry, Giulia Saccà, Junliang Shen, Paolo Stellari, Luca Tasin, Ziyu Zhang. We are very grateful to all of them. Part of this paper was written when the second author was visiting the University of California Santa Barbara, MSRI, the Max-Planck-Institut f\"ur Mathematik in Bonn, the University of Edinburgh and the University of Warwick, whose hospitality is gratefully acknowledged.

C.~L.~ is a Leverhulme Early Career Fellow and would like to acknowledge the Leverhulme Trust for the support. L.~P.~ was supported by the ERC Consolidator Grant ERC-2017-CoG-771507-StabCondEn. X.~Z.~ was partially supported by the Simons Collaborative Grant 636187.

\section{Preliminaries on stability conditions on $\Ku(Y)$} \label{section_preliminaries}

In this section we review the definitions of (weak) stability conditions on a triangulated category, semiorthogonal decompositions, and Kuznetsov component $\Ku(Y)$ of a cubic fourfold $Y$. Then we recall the construction of stability conditions on $\Ku(Y)$ due to \cite{BLMS:kuzcomponent} and some useful properties. The new contributions are an easier expression for the central charge of these stability conditions in Proposition \ref{prop:stabonku} and Lemma \ref{lem:tiltstabtokustab} which makes more clear how to check the stability of objects in $\Ku(Y)$.

\subsection{(Weak) stability conditions} It is in general a difficult task to construct stability conditions on a triangulated category. In the case of the Kuznetsov component Ku$(Y)$ of cubic fourfolds, it is proved in \cite{BLMS:kuzcomponent} that such stability conditions can be induced by `restricting' certain weak stability conditions, which can be constructed via the tilting heart technique. In this section,  we briefly recall the notion of weak stability conditions following the summary in \cite[Section 2]{BLMS:kuzcomponent}.

Let $\TT$ be a $\C$-linear triangulated category. We denote by $\kn(\TT)$ the numerical Grothendieck group of $\TT$. Let $\Lambda$ be a finite rank lattice with a surjective homomorphism $v: \kn(\TT) \twoheadrightarrow \Lambda$.

\begin{defn}\label{def:heartstructrue}
The \emph{heart of a bounded t-structure} is a full subcategory $\AA$ of $\TT$ such that 
\begin{enumerate}
    \item[(a)]  for any objects $E$ and  $F$ in $\AA$ and negative integer $n$, we have $\Hom(E,F[n])=0$;
    \item[(b)] for every $E$ in $\TT$, there exists a sequence of morphisms
$$0=E_0 \xrightarrow{\phi_1} E_1 \xrightarrow{\phi_2} \dots \xrightarrow{\phi_{m-1}} E_{m-1} \xrightarrow{\phi_m} E_m=E$$
such that the cone of $\phi_i$ is of the form $A_i[k_i]$, for some sequence $k_1 > k_2 > \dots > k_m$ of integers and objects $A_i$ in $\AA$.
\end{enumerate}
\end{defn}

Recall that the heart of a bounded t-structure is an abelian category by \cite{BBD}. 

\begin{defn}\label{def:weakstabfunction}
Let $\AA$ be an abelian category. A group homomorphism $Z: \kn(\AA) \rightarrow \C$ is a
\emph{weak stability function} (resp.\ a \emph{stability function}) on $\AA$ if, for $E \in \AA$, we have $\Im Z(E) \geq 0$, and in the case that $\Im Z(E) = 0$, we have $\Re Z(E)\leq 0$ (resp.\ $\Re Z(E) < 0$ when $E \neq 0$).
\end{defn}
For every object $E$ in  $\AA$, its slope with respect to $Z$ is given by
$$\mu_{Z}(E)= 
\begin{cases}
-\frac{\Re Z(E)}{\Im Z(E)} & \text{if } \Im Z(E) > 0, \\
+ \infty & \text{otherwise.}
\end{cases}$$
An object $E$ in  $\AA$ is \emph{semistable} (resp.\ \emph{stable}) with respect to $Z$ if for every proper subobject $F$ of $E$ in $\AA$, we have $\mu_{Z}(F) \leq \mu_{Z}(E)$ (resp.\ $\mu_{Z}(F) < \mu_{Z}(E/F))$.

\begin{defn}\label{def:wstab}
A \emph{weak stability condition} (with respect to $\Lambda$) on $\TT$ is a pair $\sigma=(\AA,Z)$, where $\AA$ is the heart of a bounded t-structure on $\TT$ and $Z$ is a group homomorphism from $\Lambda$ to $\C$, satisfying the following properties:
\begin{enumerate}
    \item [(a)] The composition $\kn(\AA)=\kn(\TT) \xrightarrow{v} \Lambda \xrightarrow{Z} \C$ is a weak stability function on $\AA$. \footnote{We will write $Z(-)$ instead of $Z(v(-))$ for simplicity.} We say that an object $E$ in $\AA[k]$ is $\sigma$-(semi)stable if $E[-k]$ is (semi)stable with respect to $Z$.
    \item [(b)] Every object of $\AA$ has a Harder--Narasimhan filtration with $\sigma$-semistable factors.
    \item [(c)] There exists a quadratic form $Q$ on $\Lambda \otimes \R$ such that the restriction of $Q$ to $\ker Z$ is negative definite and $Q(E) \geq 0$ for all $\sigma$-semistable objects $E$ in $\AA$.
\end{enumerate}
In particular, if $Z$ is a stability function, then $\sigma$ is a \emph{stability condition} introduced by Bridgeland in \cite{Bridgeland:Stab}. In this situation, we will usually call $Z$ the \emph{central charge} of the stability condition.
\end{defn}

\begin{rmk}\label{rem:wkstab}
There is usually a natural choice of the lattice $v:\kn(\TT)\rightarrow \Lambda$ in each case of triangulated categories considered in this paper.
\end{rmk}

\subsection{Semiorthogonal decompositions and Kuznetsov components} 
\begin{defn}\label{def:sod}
Let $\TT$ be a triangulated category. A \emph{semiorthogonal decomposition} 
	\begin{equation*}
	\TT = \langle \DD_1, \dots, \DD_m \rangle
	\end{equation*}
	is a sequence of full triangulated subcategories $\DD_1, \dots, \DD_m$ of $\TT$ such that: 
	\begin{enumerate}
		\item[(a)] $\Hom(F, G) = 0$, for any objects $F$ in $ \DD_i$, $G$ in $\DD_j$ and $i>j$;	
		\item[(b)] For any object $F$ in  $\DD$, there is a unique sequence of morphisms
		\begin{equation*}
		0 = F_m \to F_{m-1} \to \cdots \to F_1 \to F_0 = F,
		\end{equation*}
		with factors $\mathsf{pr}_i(F):=\mathrm{Cone}(F_i \to F_{i-1}) \in \DD_i$ for $1 \leq i \leq m$. 
	\end{enumerate}
The subcategories $\DD_i$ are called the \emph{components} of the decomposition.
\end{defn}
\begin{defn}\label{def:excobj}
An object $E$ in $\TT$ is \emph{exceptional} if $\Hom(E,E[p])=0$ for all integers $p\neq0$, and $\Hom(E,E)\cong \C$. 

A set of objects $\{E_1,\ldots,E_m\}$ in $\TT$ is an \emph{exceptional collection} if $E_i$ is an exceptional object for all $i$, and $\Hom(E_i,E_j[p])=0$ for all $p$ and all $i>j$.
\end{defn}
By \cite{BondalOrlov:Main}, an exceptional collection $\{E_1,\dots,E_m\}$ in $\TT$  provides a semiorthogonal decomposition
\begin{equation}\label{eq:sod}
    \TT = \langle \DD,E_1,\dots,E_m \rangle.
\end{equation}
Here by abuse of notation, we write $E_i$ also for the full triangulated subcategory of $\TT$ generated by $E_i$. The full subcategory $\DD:=\langle E_1,\dots,E_m\rangle^\perp$ consists of objects
\begin{equation}\label{eq:sodcubic4fold}
   \{G \in \mathrm{Obj}( \TT)| \Hom (E_i,G[p])=0 \text{ for all } p \text{ and } i\}.
\end{equation}

Let $Y$ be a smooth cubic fourfold in particular. Denote by $H$ a hyperplane section of $Y$. There is an exceptional collection $\{\OO_Y,\OO_Y(H),\OO_Y(2H)\}$.  The bounded derived category of coherent sheaves on $Y$ admits a semiorthogonal decomposition of the form 
\begin{equation}\label{eq:kuofY}
\Db(Y)=\langle\Ku(Y),\OO_Y,\OO_Y(H),\OO_Y(2H)\rangle.
    \end{equation}
The subcategory $\Ku(Y)$ is studied in details in \cite{Kuz:fourfold}, and it is now commonly referred to as the Kuznetsov component. 

%We denote the projection functor to $\Ku(Y)$ in Definition \ref{def:sod} (b) by $\mathsf{pr}$.
This projection functor in Definition \ref{def:sod} can be expressed by compositions of \emph{left (right) mutation} functors, depending on the explicit semiorthogonal decomposition. 

\begin{defn}\label{def:mutationfunctor}
Let $E$ be an exceptional object in $\TT$. The left(resp. right) mutation functors $\mathsf L_E$(resp. $\mathsf R_E$) are defined as follows:
\begin{align*}
    \mathsf L_E (F)&\coloneqq \mathrm{Cone}\left(\bigoplus_{p\in \Z}\Hom(E[p],F)\otimes E[p]\xrightarrow {\mathsf{ev}} F\right);\\
    \mathsf R_E(F) &\coloneqq \mathrm{Cone}\left(F\xrightarrow{\mathsf{ev}^\vee}\bigoplus_{p\in \Z}\Hom(F,E[p])^\vee\otimes E[p]\right)[-1].
\end{align*}
\end{defn}
%In the setup of the semiorthogonal decomposition \eqref{eq:sod}, the projection functor $\mathrm{pr}_\DD=\mathsf L_{E_1}\dots \mathsf L_{E_m}$. 
%\end{rmk}

\begin{rmk}\label{rem:mutations}
Note that since the Serre functor in $\Db(Y)$ is given by $$\mathsf S_{Y}(-)=-\otimes \OO_Y(-3H)[4],$$ the Kuznetsov component of $Y$ can be also given by $^\perp\!\langle \OO_Y(-2H)\OO_Y(-H)\rangle\cap \OO_Y^\perp$. Namely, it appears in the semiorthogonal decomposition
\begin{equation}\label{eq:kuinthemidsod}
    \langle \OO_Y(-2H), \OO_Y(-H),\Ku(Y),\OO_Y\rangle.
\end{equation}
Note that the object $\mathsf L_E (F)$ (resp. $\mathsf R_E(F)$) is in $E^\perp$(resp. $^\perp\! E$). We denote by  $$\mathsf{pr}:\Db(Y)\rightarrow \Ku(Y)$$  the functor as that in Definition \ref{def:sod}(b) to the component $\Ku(Y)$ with respect to the decomposition \eqref{eq:kuinthemidsod}. In particular, it is given as the composition of mutations:
\begin{equation}\label{eq:defofpr}
    \mathsf{pr}= \mathsf R_{\OO_Y(-H)}\mathsf R_{\OO_Y(-2H)}\mathsf L_{\OO_Y} = \mathsf L_{\OO_Y}\mathsf R_{\OO_Y(-H)}\mathsf R_{\OO_Y(-2H)}.
\end{equation}
We will use the functor $\mathsf{pr}$ to produce objects in $\Ku(Y)$ in Section \ref{sec:stabobjinM}. Note that the functor $\mathsf{pr}$ in our paper is different from the more standard projection functor $\mathsf L_{\OO_Y}\mathsf L_{\OO_Y(H)}\mathsf L_{\OO_Y(2H)}$ with respect to the decomposition  \eqref{eq:sodcubic4fold}, which is also the left adjoint functor of the natural embedding of $\Ku(Y)$.

\end{rmk}

\subsection{Kuznetsov components of $\Db(Y)$ and $\Db(\P^3,\BB_0)$}  \label{subsec_kuzcompoinP3}
It is usually a highly non-trivial task to construct stability conditions on the Kuznetsov component. The only technique so far is to restrict weak stability conditions on the whole derived category to its Kuznetsov component. In the cubic fourfold case, such weak stability conditions on $\Db(Y)$ require a Bogomolov type inequality involving the third Chern character. Unfortunately, such inequality is not known yet for any cubic fourfold. 

To avoid this technical difficulty, the idea in \cite{BLMS:kuzcomponent} is to embed $\Ku(Y)$ as a component in a bounded derived category of lower dimension. More precisely, the key observation in \cite{BLMS:kuzcomponent} is that $\Ku(Y)$ is equivalent to the Kuznetsov component of $\Db(\P^3,\BB_0)$. We briefly summarize the construction of this equivalence in this section.

Let $L\subset Y\subset \P^5$ be a line which is not on any plane in $Y$, and we denote by $$\rho_L:\tilde Y\rightarrow Y$$ the blow-up of $L$ in $Y$. 

 The projection from $L$ to a disjoint $\P^3$ (in $\P^5$) equips $\tilde{Y}$ with a natural conic fibration structure
$$\pi:\tilde Y \to \P^3.$$
There is a rank three vector bundle $\FF \cong \OO_{\P^3}^{\oplus 2}\oplus \OO_{\P^3}(-1)$ on $\P^3$ such that $\tilde Y$ embeds into the $\P^2$-bundle $\P_{\P^3} (\FF)$ as the zero locus of a section
$$s_{\tilde Y}\in \mathrm{H}^0(\P^3,\mathrm{Sym}^2\FF^\vee\otimes \OO_{\P^3}(1))\cong \mathrm{H}^0(\P_{\P^3}(\FF),\OO_{\P_{\P^3}(\FF)}(2)\otimes q^*\OO_{\P^3}(1)).$$
The cartoon of these morphisms is as follow:
\begin{equation}\label{diag:blowupL}
\xymatrix{ \tilde Y \ar[d]_-{\rho_L} \ar[rrd]^-{\;\;\pi} \ar[r]^{\alpha} & \mathrm{Bl}_L\P^5 \ar@{=}[r] \ar[d] & \P_{\P^3}(\FF)  \ar[d]^q\\
Y \ar[r] & \P^5 &  \P^3.
}
\end{equation}

By \cite[Section 3]{Kuz:Quadric}, we have an associated sheaf of Clifford algebras of $\pi$ over $\P^3$. Denote its even part (resp.\ odd part) by $\BB_0$ (resp.\ $\BB_1$).  By \cite[(12)]{Kuz:Quadric}, as a sheaf on $\P^3$, the even part $\BB_0$ is a rank four vector bundle:
$$\OO_{\P^3}\oplus \left(\bigwedge \! ^2 \FF\otimes \OO_{\P^3}(-1)\right)\cong \OO_{\P^3}\oplus  \OO_{\P^3}(-1)\oplus \OO_{\P^3}(-2)^{\oplus 2}.$$

As for its algebra structure, the structure sheaf is central. The other relations are determined by
\begin{equation}\label{eq:clifalg}
    e_i\wedge e_k\cdot e_k\wedge e_j= s_{\tilde Y}(e_k\otimes e_k) e_i\wedge e_j,\;\;\;\; e_i\wedge e_k\cdot e_i\wedge e_k= s_{\tilde Y}(e_i\otimes e_i) s_{\tilde Y}(e_k\otimes e_k),
\end{equation}
for an orthogonal basis $(e_1,e_2,e_3)$ and $i\neq j\neq k\neq i$.

\begin{defn}\label{def:p3b0alg}
We denote by $\Coh(\P^3,\BB_0)$ the category of coherent sheaves on $\P^3$ with a right $\BB_0$-module structure, and denote its bounded derived category by $\Db(\P^3,\BB_0)$. The natural forgetful functor is denoted by $\mathrm{Forg}: \Db(\P^3,\BB_0)\rightarrow \Db(\P^3)$.
\end{defn}
 In particular, we have 
\begin{equation}\label{eq:forgetful}\Hom_{\Db(\P^3,\BB_0)}(\BB_0,\GG)\cong \Hom_{\Db(\P^3)}(\OO_{\P^3},\mathrm{Forg}(\GG))
    \end{equation} for every $\GG\in \Db(\P^3,\BB_0)$.
    
By \cite[(14)]{Kuz:Quadric}, the odd part $\BB_1$  as a coherent sheaf is $$\mathrm{Forg}(\BB_1)=\mathrm{Forg}(\FF\oplus(\bigwedge\!^3\FF\otimes \OO_{\P^3}(-1)))\cong \OO_{\P^3}^{\oplus 2}\oplus  \OO_{\P^3}(-1)\oplus \OO_{\P^3}(-2).$$    
As in \cite[(15)]{Kuz:Quadric}, we define the following $\BB_0$-bimodules for $j\in \Z$:
$$\BB_{2j}:= \BB_0 \otimes \OO_{\P^3}(j) \quad \text{and} \quad  \BB_{2j+1}:= \BB_1 \otimes \OO_{\P^3}(j).$$
The Serre functor on $\Db(\P^3,\BB_0)$ (see \cite[page 28]{BLMS:kuzcomponent}) is explicitly given as
\begin{equation} \label{eq:serrep3b0}
\mathsf S_{\BB_0}(-) = (-)\otimes_{\BB_0}\BB_{-3}[3].
\end{equation}
By \cite[Lemma 3.8 and Corollary 3.9]{Kuz:Quadric} and a direct computation using \eqref{eq:forgetful} and \eqref{eq:serrep3b0}, the ordered set $\{\BB_1,\BB_2,\BB_3\}$ is an exceptional collection  in $\Db(\P^3,\BB_0)$. By \eqref{eq:sod}, there is a semiorthogonal decomposition of the form 
\begin{equation}
\label{eq:kuzofp3b0}
\Db(\P^3,\BB_0)=\langle \Ku(\P^3,\BB_0), \BB_1, \BB_2, \BB_3 \rangle,
\end{equation}
One of the key observations in \cite[Section 7]{BLMS:kuzcomponent} to construct stability conditions on $\Ku(Y)$ is as follows.
\begin{prop}[{\cite[Lemma 7.6 and Proposition 7.7]{BLMS:kuzcomponent}}]\label{prop:equivalentofKu}
The Kuznetsov component $\Ku(Y)$ is equivalent to $\Ku(\P^3,\BB_0)$.
\end{prop}
\begin{rmk}\label{rem:equivalentofKu}
The functor of this equivalence is given explicitly as
\begin{align}
    \Psi\circ\rho_L^*:\Ku(Y)\hookrightarrow \Db(Y) \xrightarrow{\rho_L^*} \Db(\tilde Y) \xrightarrow{\Psi} \Db(\P^3,\BB_0)\supset \Ku(\P^3,\BB_0), 
\end{align}
where the functor $\Psi$ is defined by 
$$\Psi(-)=\pi_*(- \otimes \EE[1]).$$
Here $\EE$ is a sheaf of right $\pi^*\BB_0$-modules on $\tilde Y$ as that defined in \eqref{eq:defofcE}.  We will only make essential use of this functor in Section \ref{sec:stabE}.
\end{rmk}

\subsection{Weak stability conditions on $\Db(\P^3,\BB_0)$}
We first review the notion of t-structure by tilting. Let $\sigma=(\AA,Z)$ be a weak stability condition on $\TT$, and $t\in\R$. We can form
the following subcategories of $A$:
\begin{align*}
    \AA_\sigma^{>t}\coloneqq \{E|\text{ every Harder--Narasimhan factor $F$ of $E$ has $\mu_Z(F)>t$}\};\\
     \AA_\sigma^{\leq t}\coloneqq \{E|\text{ every Harder--Narasimhan factor $F$ of $E$ has $\mu_Z(F)\leq t$}\}.
\end{align*}
It follows from the existence of Harder--Narasimhan filtrations that $(\AA_\sigma^{>\mu},\AA_\sigma^{\leq\mu})$ forms a \emph{torsion pair} in $\AA$
in the sense of \cite{Happel-al:tilting}. In particular, we can obtain a new heart of a bounded t-structure by tilting.
\begin{PropDef}[\cite{Happel-al:tilting}]\label{pd:tiltheart}
Given a weak stability condition $\sigma = (\AA,Z)$ on $\TT$ and a choice of slope $t\in\R$, there exists a heart of a bounded t-structure defined by:
$$\AA_\sigma^t\coloneqq\langle \AA_\sigma^{>t},\AA_\sigma^{\leq t}[1]\rangle_{\text{extension closure}}.$$
\end{PropDef}

For an object $F$ in $\Db(\P^3,\BB_0)$, we define its modified Chern character as
\begin{equation} \label{eq:modchern}
    \ch_{\BB_0}(F)=\ch(\text{Forg}(F))(1-\frac{11}{32}l),
\end{equation}
where $l$ denotes the class of a line in $\P^3$. 
For every $\beta\in \R$, we define the twisted Chern character as 
\[\ch_{\BB_0}^\beta=e^{-\beta h}\ch_{\BB_0}=(\rk,\ch_{\BB_0,1}-\rk \beta h,\ch_{\BB_0,2}-\beta h \cdot \ch_{\BB_0,1}+\rk \frac{\beta^2}{2} h^2,\dots),\]
where $h$ denotes the class of a plane in $\P^3$.

We fix the lattice $\Lambda$ of $\kn(\P^3,\BB_0)$ in Definition \ref{def:wstab} as:
$$v=\ch_{\BB_0,\leq 2}:\kn(\P^3,\BB_0)\rightarrow \Lambda =\{(\rk(F),l\cdot\ch_1(\Forg(F)),h\cdot\ch_2(\Forg(F))-\frac{11}{32}\rk(F))\}.$$ To simplify the notation, we will also write $\ch_1$(resp. $\ch_2$) for $l\cdot\ch_1$(resp. $h\cdot \ch_2$) in the future when there is no ambiguity. 

The \emph{discriminant} of an object $F$ in $\Db(\P^3,\BB_0)$ is defined as follows:
\begin{equation}\label{eq:discrim}
    \dlt(F)\coloneqq(\ch_{\BB_0,1}(F))^2-2\rk(F)\ch_{\BB_0,2}(F) = (\ch^{\beta}_{\BB_0,1}(F))^2-2\rk(F)\ch^{\beta}_{\BB_0,2}(F).
\end{equation}

Setting $Z_{\mathrm{slope}}=i\rk-\ch_{\BB_0,1}$, then the classical slope stability $$\sigma_{\mathrm{slope}}=(\Coh(\P^3,\BB_0),Z_{\mathrm{slope}})$$ is a weak stability condition in the sense of Definition \ref{def:wstab}. For $\beta\in \R$, as in Proposition and Definition \ref{pd:tiltheart}, we have the heart of bounded t-structure
$$\mathrm{Coh}^\beta(\P^3, \BB_0)\coloneqq\langle\mathrm{Coh}^{>\beta}_{\sigma_\mathrm{slope}}(\P^3, \BB_0),\mathrm{Coh}^{\leq \beta}_{\sigma_\mathrm{slope}}(\P^3, \BB_0)[1]\rangle.$$

With this notation, we can state the following result.

\begin{PropDef}[{\cite[Proposition 9.3]{BLMS:kuzcomponent}}]\label{prop:constructstabcond}
Given $\alpha>0$ and $\beta\in \R$, the pair $\sigma_{\alpha,\beta}=(\mathrm{Coh}^\beta(\P^3, \BB_0), \zab)$ with
$$\zab(F)= i\ch^{\beta}_{\BB_0,1}(F)+\frac{1}2\alpha^2\ch^{\beta}_{\BB_0,0}(F)-\ch^{\beta}_{\BB_0,2}(F)$$
defines a weak stability condition on $\Db(\P^3,\BB_0)$.  These stability conditions vary continuously as $(\alpha,\beta)\in \R_{>0}\times \R$ varies. The quadratic form in Definition \ref{def:wstab} can be taken as $\dlt$. In particular, if an object $F$ is $\sigma_{\alpha,\beta}$-semistable for some $\alpha>0,\beta\in \R$, then we have
\[\dlt(F)\geq 0.\]
\end{PropDef}

\begin{rmk}\label{rem:wkstabonp3b0}
\begin{enumerate}
    \item [(i)] By `vary continuously', we mean that if an object $F$ is $\sigma_{\alpha_0,\beta_0}$-stable for some $\alpha_0>0$ and $\beta_0 \in \R$, then $F$ is $\sigma_{\alpha,\beta}$-stable for $(\alpha,\beta)$ in an open neighbourhood of $(\alpha_0,\beta_0)$.
    \item[(ii)] For all $j\in \Z$, the object $\BB_j$ is $\sigma_{\alpha,\beta}$-stable for every $\alpha>0$ and $\beta\in \R$.
\end{enumerate}
\end{rmk}

\subsection{Stability conditions on $\Ku(\P^3,\BB_0)$ and $\Ku(Y)$}
\label{subsection_stabcondonKu(Y)}
The weak stability conditions in Proposition \ref{prop:constructstabcond} do not restrict to stability conditions on $\Ku(\P^3,\BB_0)$ directly. We need to modify them by one more tilting. 

Fix some $\alpha<\frac{1}4$ and $\beta=-1$. Consider the tilting of  $\sigma_{\alpha,-1}$ with respect to the slope value $0$ as in Proposition and Definition \ref{pd:tiltheart}. We refurbish the main result for the stability conditions on $\Ku(\P^3,\BB_0)$ in \cite[Theorem 1.2]{BLMS:kuzcomponent} as follows.

\begin{PropDef}\label{prop:stabonku}
The pair
 \begin{equation}
\label{stabcond}
\sigma_\alpha\coloneqq \left(\left(\Coh^{-1}(\P^3, \BB_0)\right)^0_{\sigma_{\alpha,-1}}\bigcap \Ku(\P^3,\BB_0),Z=\ch^{-1}_{\BB_0,1}-i\rk\right)
\end{equation}
is a stability condition on $\Ku(\P^3,\BB_0)$ with respect to the natural rank-$2$ lattice
\begin{equation*}
    \ch_{\BB_0,\leq 1}:\kn(\Ku(\P^3,\BB_0))\rightarrow \Lambda = \{(\rk,\ch_1)\}.
\end{equation*}
Moreover, the stability condition $\sigma_\alpha$ does not depend on the choice of $\alpha$. 
\end{PropDef}
\begin{proof}
By Proposition \ref{prop:constructstabcond} and \cite[Proposition 2.15, Proposition 5.1 and Proof of Theorem 1.2]{BLMS:kuzcomponent}, if we replace the central charge in the pair \eqref{stabcond} by $-iZ_{\alpha,-1}$, then that will be a stability condition on $\Ku(\P^3,\BB_0)$. We only need to check that  the  central charge in \eqref{stabcond} induces the same slope function as that induced by $-iZ_{\alpha,-1}$.

For any object $E$ in $\Ku(\P^3,\BB_0)$, by \eqref{eq:sodcubic4fold} and \eqref{eq:kuzofp3b0} , we have 
\begin{equation}\label{eq:chivan}
    \chi_{\Db(\P^3,\BB_0)}(\BB_1,E)=\chi_{\Db(\P^3,\BB_0)}(\BB_3,E)=0.
\end{equation}
Recall that the Riemann--Roch formula for $\P^3$ is given as \begin{equation}\label{eq:rrp3}
    \chi_{\P^3}(F)=\ch_3(F)+2\ch_2(F)+\frac{11}6\ch_1(F)+\rk(F)
\end{equation}for every object $F$ in $\Db(\P^3)$.

Denote $G=\Forg(E\otimes_{\BB_0}\BB_{-1})$. By \eqref{eq:forgetful}, \eqref{eq:chivan} and \cite[Corollary 3.9]{Kuz:Quadric}, we have
\begin{align}\label{eqn:tosubtract}
    0=&\chi_{\Db(\P^3,\BB_0)}(\BB_1,E)=\chi_{\P^3}(G)= \ch_3(G)+2\ch_2(G)+\frac{11}6\ch_1(G)+\rk(G),\\ \label{eqn:tosubtract2} 0=&\chi_{\Db(\P^3,\BB_0)}(\BB_3,E)=\chi_{\P^3}(G(-H)) = \ch_3(G)+\ch_2(G)+\frac{1}3\ch_1(G).
\end{align}

By \cite[Corollary 3.9]{Kuz:Quadric}, for every $\BB_i$ we have
\begin{align*}
    \ch_2(\BB_{i}\otimes_{\BB_0}\BB_{-1})=\ch_2(\BB_{i-1})=\ch_2(\mathrm{Forg}(\BB_i))-\frac{1}2\ch_1(\mathrm{Forg}(\BB_i))+\frac{1}8\rk(\BB_i).
\end{align*}
By \cite[Proposition 2.12]{BMMS:Cubics} and restricting to a general hyperplane, the character $\ch_{\leq2}(E)$ is spanned by $\ch_{\leq 2}(\BB_i)$'s. As the sheaf $\BB_{-1}$ is a flat $\BB_0$-module, the operation $-\otimes_{\BB_0}{\BB_{-1}}$ is linear on $\ch_{\leq 2}$, so we have 
\begin{align*}
    \ch_2(G)=\ch_2(\mathrm{Forg}(E))-\frac{1}2\ch_1(\mathrm{Forg}(E))+\frac{1}8\rk(E)\text{ and }
    \ch_1(G)=  \ch_1(\mathrm{Forg}(E))-\frac{1}2\rk(E).
\end{align*}%\todo{How are $\ch_2(G)$ and $\ch_1(G)$ computed? Check if this is OK now. It is great, thanks!}
By subtracting the equations \eqref{eqn:tosubtract} and \eqref{eqn:tosubtract2}, we have 
\begin{equation}\label{eqn:charofKuobj}
 \ch_2(\mathrm{Forg}(E))=-\ch_1(\mathrm{Forg}(E))-\frac{3}8\rk(E).
\end{equation}
Note that
\begin{align*}
   \Im(-iZ_{\alpha,-1}(E))=&\ch_{2}^{-1}(\Forg E)-\frac{11}{32}\rk(E)-\frac{1}2\alpha^2\rk(E)\\ =&\ch_{2}(\Forg E)+\ch_{1}(\Forg E)+\frac{1}2\rk(E)-\frac{11}{32}\rk(E)-\frac{1}2\alpha^2\rk(E) \\ =&-(\frac{7}{32}+\frac{1}2\alpha^2)\rk(E) \qquad \text{ by \eqref{eqn:charofKuobj}.}
\end{align*}
We have
\begin{align*}
    Z(E)= \ch^{-1}_{\BB_0,1}(E)-i\rk(E) =\Re(-iZ_{\alpha,-1}(E)) +\frac{i}{\frac{7}{32}+\frac{1}2\alpha^2}\Im(-iZ_{\alpha,-1}(E)).
\end{align*}
Therefore, the slope function induced by $-iZ_{\alpha,-1}$ is the same  up to a constant scalar as that of $Z=\ch^{-1}_{\BB_0,1}-i\rk$. So the pair \eqref{stabcond} is a stability condition on $\Ku(\P^3,\BB_0)$. Note that the stability conditions $\sigma_\alpha$'s vary continuously when $\alpha$ varies. Since  they have the same central charge, all of them are the same stability condition.
\end{proof}
It is worth pointing out that to check the stability of an object in $\Ku(\P^3,\BB_0)$, we usually only need to work in the heart $\Coh^{-1}(\P^3,\BB_0)$. The following simple lemma makes this more precise. 
\begin{lem}\label{lem:tiltstabtokustab}
Let $E$ be a $\sigma_{\alpha,-1}$-stable object such that \begin{enumerate}
    \item[(a)] $E$ is an  object in $\left(\Coh^{-1}(\P^3, \BB_0)\right)^0_{\sigma_{\alpha,-1}}\bigcap \Ku(\P^3,\BB_0)$;
    \item[(b)] $\Hom_{\BB_0}(T,E)=0$ for every $T\in \Coh(\P^3,\BB_0)$ supported on zero dimensional locus.
\end{enumerate}
Then $E$ is $\sigma_{\alpha}$-stable. 
\end{lem}
\begin{proof}
We only need to show that $E$ is stable with respect to the weak stability condition $\left(\left(\Coh^{-1}(\P^3, \BB_0)\right)^0_{\sigma_{\alpha,-1}}, -iZ_{\alpha,-1}\right)$. Denote $\AA=\Coh^{-1}(\P^3,\BB_0)$. Let $F$ be a non-zero proper subobject of $E$ in $\left(\Coh^{-1}(\P^3, \BB_0)\right)^0_{\sigma_{\alpha,-1}}$, then we have the exact sequence of objects in $\AA$:
\begin{equation}
    0\rightarrow \HH^{-1}_\AA(F)\rightarrow \HH^{-1}_\AA(E)\rightarrow \HH^{-1}_\AA(E/F)\rightarrow \HH^{0}_\AA(F)\rightarrow \HH^{0}_\AA(E)\rightarrow \HH^{0}_\AA(E/F)\rightarrow 0.
\end{equation}
Since $E$ is $\sigma_{\alpha,-1}$-stable, either $\HH^{-1}_\AA(E)$ or $\HH^{0}_\AA(E)=0$. When $\HH^{0}_\AA(E)\neq 0$, we have $\HH^{-1}_\AA(F)=\HH^{-1}_\AA(E)=0$. Since $E=\HH^{0}_\AA(E)$ is $\sigma_{\alpha,-1}$-stable, we have
\begin{equation}
\mu_{Z_{\alpha,-1}}(\HH^{0}_\AA(E/F))\geq \mu_{Z_{\alpha,-1}}(E)>0. 
\end{equation}
By condition (b), we have $Z_{\alpha,-1}(F)\neq 0$. Since $\mu_{Z_{\alpha,-1}}(\HH^{-1}_\AA(E/F))<0$ or $\HH^{-1}_\AA(E/F)=0$,  it is clear that 
\begin{equation}\label{2.14}
    \mu_{Z_{\alpha,-1}}(E)\geq \mu_{Z_{\alpha,-1}}(F)>0.
\end{equation}
The equality in \eqref{2.14} can hold only when $\HH^{-1}_\AA(E/F)= 0$ and $Z_{\alpha,-1}(E/F)=0$. In particular,  $E/F=\HH^{0}_\AA(E/F)$ in this case.  By Definition \ref{def:weakstabfunction}, we always have $$\mu_{Z_{\alpha,-1}}(E/F) \text{ or }\mu_{Z_{\alpha,-1}}(E)>\mu_{Z_{\alpha,-1}}(F)>0.$$
Therefore, $$\mu_{-iZ_{\alpha,-1}}(E/F)\text{ or }\mu_{-iZ_{\alpha,-1}}(E)>\mu_{-iZ_{\alpha,-1}}(F)>0.$$
By Definition \ref{def:weakstabfunction}, the object $E$ is $\mu_{-iZ_{\alpha,-1}}$-stable.  A similar argument also holds for the case when $\HH^{-1}_\AA(E)\neq 0$. 
\end{proof}

Another subtle issue is that the Clifford structure and the embedding of $\Ku(Y)$ in $\Db(\P^3,\BB_0)$ depend on the choice of the line $L$ to blow up, see Remark \ref{rem:equivalentofKu}. However, for the induced stability conditions on the Kuznetsov component, we have the following result.

\begin{prop}[{\cite[Proposition 2.6]{LPZ1}}]
\label{prop:change_line}
If $\sigma$ is a stability condition as defined in \eqref{stabcond}, then the induced stability condition $(\Psi\circ \rho_L^*)^{-1}\sigma$ on $\Ku(Y)$  is independent of the choice of $L$.
\end{prop}

\begin{rmk}\label{rem:noL}
As we are only interested in the stability of objects in $\Ku(Y)$, we will omit $L$ in all the morphisms and functors that rely on $L$ in the future texts.  We will also write $\sigma$ instead of $(\Psi\circ \rho_L^*)^{-1}\sigma$ for the stability condition on $\Ku(Y)$ for simplicity.

The stability condition $\sigma$ is also a \emph{full numerical stability condition} with respect to the whole lattice $\kn(\Ku(Y))$. The whole connected component $\Stab^\dag(\Ku(Y))$ containing $\sigma$ is described as that in \cite[Theorem 29.1]{BLM+}.
\end{rmk}

\section{Symplectic resolution of the moduli space $M_\tau(2v_0)$ with $v_0^2=2$} \label{sec_sympleresol}
This section is devoted to the proof of Theorem \ref{thm_OG10intro}. After recalling the definition of algebraic Mukai lattice of $\Ku(Y)$ and stating the main result, we describe the local structure of the moduli space $M_\tau(2v_0)$ at the worst singular points. This is used to construct the symplectic resolution $\tM$ by blowing up the singular locus with the reduced scheme structure as in \cite{LS}. Finally, we obtain the projectivity and the deformation class of $\tM$ by specializing to Kuznetsov components equivalent to the bounded derived category of a K3 surface.

\subsection{Algebraic Mukai lattice of $\Ku(Y)$} \label{subsec_algMukai} Let $Y$ be a cubic fourfold over $\C$ and $\Ku(Y)$ be its Kuznetsov component. The \emph{algebraic Mukai lattice} $H^*_{\mathrm{alg}}(\Ku(Y),\Z)$ of $\Ku(Y)$  is introduced in \cite[Proposition and Definition 9.5]{BLMS:kuzcomponent}. It consists of algebraic cohomology classes of $Y$ which are orthogonal to the classes of $\OO_Y$, $\OO_Y(H)$, $\OO_Y(2H)$ with respect to the Euler pairing. 

As for an alternative description, the algebraic Mukai lattice is $\kn(\Ku(Y))$ equipped with a Mukai pairing:
\begin{equation}\label{eq:mukaipairing}
    ([E],[F])\coloneqq -\chi(E,F) = -\chi(F,E)
\end{equation}
for objects $E$ and $F$ in $\Ku(Y)$. The signature of the Mukai pairing is $(2,\rho)$, where $0\leq \rho \leq 20$.

We will be only interested in a sub-lattice in $\kn(\Ku(Y))$ generated by  two special classes
\begin{equation}\label{eq:defoflambdas}
\lambda_i\coloneqq[\mathsf{L}_{\OO_Y}\mathsf{L}_{\OO_Y(H)}\mathsf{L}_{\OO_Y(2H)}(\OO_L(iH))] \text{ for $i=1$ and $2$},
\end{equation}
where $L$ is a line on $Y$ and $\mathsf{L}_{\OO_Y}\mathsf{L}_{\OO_Y(H)}\mathsf{L}_{\OO_Y(2H)}:\Db(Y)\to \Ku(Y)$ is the  projection functor with respect to the semiorthogonal decomposition $\langle\Ku(Y),\OO_Y,\OO_Y(H),\OO_Y(2H)\rangle$ of $\Db(Y)$. The Mukai pairing of them can be computed as:
\begin{equation}\label{eq:mukaipairoflambdas}
    (\lambda_1,\lambda_1)=(\lambda_2,\lambda_2)=2, (\lambda_1,\lambda_2)=-1.
\end{equation}
In particular, when $Y$ is a very general cubic fourfold, the Mukai lattice is spanned by $\lambda_1$ and $\lambda_2$.

\subsection{The main theorem}
 Fix an element $v_0$ in the algebraic Mukai lattice of $\Ku(Y)$ such that $(v_0,v_0)=2$ and set $v:=2v_0$. Let $\Stab^\dag(\Ku(Y))$ be  the connected component of full numerical stability conditions on $\Ku(Y)$  containing $\sigma$.  By \cite[Theorem 21.24]{BLM+} (which makes use of the main result in \cite{AHLH}), for every $\tau \in \Stab^\dag(\Ku(Y))$, the moduli stack $\MM_{\tau}(v)$ parametrizing $\tau$-semistable objects in $\Ku(Y)$ admits a good moduli space $M_\tau(v)$, which is a proper algebraic space.

In this section, we fix a stability condition  $\tau \in \Stab^\dag(\Ku(Y))$ which is \emph{generic} with respect to $v$. In other words, the strictly $\tau$-semistable objects in $M_\tau(v)$ are S-equivalent to the direct sum of two $\tau$-stable objects with Mukai vector $v_0$. Note that this stability condition may be different as that in Remark \ref{rem:noL}, which was denoted by $\sigma$, when $Y$ is not very general. Set $M:=M_\tau(v)$; the aim of this section is to prove the following result.
\begin{thm}
\label{thm_OG10}
The moduli space $M$ has a symplectic resolution $\tM$, which is a $10$-dimensional smooth projective hyperk\"ahler manifold, deformation equivalent to the O'Grady's example constructed in \cite{OGrady}. 
\end{thm}

The construction of the symplectic resolution is done in \cite{LS} in the case of the moduli space of semistable sheaves having Mukai vector $2v_0$ with $(v_0, v_0)=2$ over a polarized K3 surface. A large part of their argument applies to our more general setup without much change. For this reason, we only sketch the proof, referring to \cite{LS} for a complete discussion. The main difference is that in the case of moduli of sheaves, the moduli is constructed as a GIT quotient. To study the local structure, it is enough to take an \'etale slice. In our case, we instead use the deep result on \'etale slice of algebraic stacks \cite{AlperHallRydh}, and we give the details for this part of the proof. 

The strategy is to study the local structure of the moduli space at the worst singularity and prove that its normal cone is isomorphic to an affine model obtained as a nilpotent orbit in the symplectic Lie algebra $\mathfrak{sp}(4)$. It turns out that the singularity is formally isomorphic to its normal cone. Since the singularity at the generic point of the singular locus of $M$ is of type $A_1$, one can conclude that the blow up $\tM$ of $M$ at its singular locus endowed with the reduced scheme structure is a symplectic resolution of $M$. The other properties of $\tM$ (projectivity, deformation type) will be obtained by degeneration to the locus of cubic fourfolds with Kuznetsov component equivalent to the bounded derived category of a K3 surface, as in \cite{BLM+}.

\subsection{Local structure of $M$}
\label{subsec_localstructuremoduli}
We have the following possibilities for $E \in M$:
\begin{enumerate}
\item $E$ is $\tau$-stable. Its automorphism group is $\Aut(E)\cong \C^*$.  
\item $E$ is S-equivalent to $F \oplus F'$ with non-isomorphic $F, F' \in M_\tau(v_0)$. In this case, we have $\Aut(E)\cong \C^* \times \C^*$.
\item $E$ is S-equivalent to $F^{\oplus 2}$ for $F \in M_\tau(v_0)$. Then, $\Aut(E)\cong \mbox{GL}(2,\C)$. 
\end{enumerate}
In this section, we investigate the structure of $M$ in a formal neighborhood of a semistable point as in item (3). 

Let $E$ be a $\tau$-semistable object in $M$. As in \cite{LS}, the first ingredient for the proof is the description of the infinitesimal deformation of $E$. In the case of polystable sheaves on a K3 surface a good summary of the results is provided in \cite[Sections 2 and 4]{ArbSacca}, which we follow in our case. The deformation theory for perfect complexes in the derived category has been studied in \cite{Lieblich:mother-of-all}. In our setting, we consider the functor
$$\Def_E: \text{Art} \to \text{Sets}$$
from the category of local Artinian $\C$-algebras to the category of sets, which assigns to an object $A$ in $\text{Art}$, the set $\Def_E(A)$ of equivalence classes of deformations of $E$ to $Y_A:=Y \times \Spec A$. Explicitly, objects in $\Def_E(A)$ are equivalence classes of pairs $(E_A,\varphi)$, where $E_A$ is a complex on $Y_A$ together with an isomorphism $\varphi: E_A \otimes^{\L}_A \C \cong E$ (see \cite[Definition 3.2.1]{Lieblich:mother-of-all}). Two pairs $(E_A,\varphi)$ and $(E_A',\varphi')$ are equivalent if there is an isomorphism $\psi: E_A \cong E'_A$ such that $\varphi' \circ \psi=\varphi$. 

Note that by \cite[Lemma 3.2.4]{Lieblich:mother-of-all},   $E_A$ is an object in $\Db(Y_A)$. By base change and the definition of $E_A$, if $p$ is the closed point of $\Spec A$, then 
$$\RlHom(\OO_{Y_A}(iH), E_A)_p \cong \RlHom_p(\OO_{Y_A}(iH)_p, E_{A \, p})_p \cong \RlHom_p(\OO_Y(iH), E)=0$$
for $i=0, 1, 2$. So the property of being in $\Ku(Y)$ is an open condition, and we may assume $E_A$ is an object in $\Ku(Y_A):=\langle \OO_{Y_A}, \OO_{Y_A}(H), \OO_{Y_A}(2H) \rangle^{\perp}$, where $\OO_{Y_A}(H)$ is the trivial deformation of $\OO_Y(H)$ to $Y_A$. By \cite[Theorem 3.1.1 and Proposition 3.5.1]{Lieblich:mother-of-all}, the functor $\Def_E$ is a deformation functor and its tangent space $\Def_E(\C[\epsilon])$ is $\Ext^1(E,E)$, where $\C[\epsilon]:=\C[t]/(t^2)$. 

As proved in \cite{Kuznetsov2009Symplectic}, the definition of the trace map requires an additional step. Denote by $\epsilon_E$ the \emph{linkage class} of $E$ (see \cite[Proposition 3.1]{Kuznetsov2009Symplectic}) %$\text{At}(E)$ the Atiyah class of $E$ 
and consider the composition
$$\text{tr}: \Ext^2(E,E) \xrightarrow{\epsilon_E \circ -} \Ext^4(E,E \otimes \Omega^4_Y) \xrightarrow{\mathrm{Tr}} H^4(\Omega^4_Y) \cong \C,$$
where the first map is given by the composition with $\epsilon_E$ and the second map is the usual trace map. We set $\Ext^2(E,E)_0:=\ker(\text{tr})$. 

Using the fact that $E \in \Db(Y)$ has an injective resolution which is equivariant with respect to the canonical action of the automorphism group $\Aut(E)$ of $E$, the argument in \cite[Appendix]{LS} allows to construct a formal map
$$\kappa=\kappa_2+\kappa_3+\dots: \Ext^1(E,E) \to \Ext^2(E,E)_0$$
known as the Kuranishi map, defined inductively on the order, with the following properties:
\begin{enumerate}
\item The map $\kappa$ is equivariant with respect to the conjugation action of $\Aut(E)$.
\item The second order term $\kappa_2: \Ext^1(E,E) \to \Ext^2(E,E)_0$ is given by the Yoneda product $\kappa_2(e)=e \mathbin{\smile} e$ for $e \in \Ext^1(E,E)$.
\item By \cite{Rim} there exists an $\Aut(E)$-equivariant formal deformation $(\widehat{E},\widehat{\varphi})$ of $E$ having the versality property, parametrized by the formal scheme $D_\kappa:=\kappa^{-1}(0)$.
\end{enumerate} 
Denote by $A:=\C[\Ext^1(E,E)]$ the polynomial ring on $\Ext^1(E,E)$.  Let $\widehat{A}$ be the completion of the ring $A$ with respect to the maximal ideal $\mathfrak{m}$ of polynomial functions vanishing at $0$. The Kuranishi map can be also written dually as
$$\kappa^*: \Ext^2(E,E)_0^* \to \mathfrak{m}^2\widehat{A}.$$
If $\mathfrak{a} \subset \widehat{A}$ is the ideal generated by the image of $\kappa^*$, then by definition we have
$$D_\kappa=\text{Spf}(\widehat{A}/\mathfrak{a})=\text{colim}_{n}\Spec((\widehat{A}/ \mathfrak{a})/\mathfrak{m}^n),$$
where $\mathfrak{m}$ is the maximal ideal of $\widehat{A}/\mathfrak{a}$ by abuse of notation.

On the other hand, the object $E$ defines a closed point $x$ in the moduli stack $\MM:=\MM_\sigma(v)$ and the S-equivalence class of $E$ determines a point $\pi(x) \in M$, where $\pi: \MM \to M$ is  the good moduli space. The stabilizer $G_x$ of $x$ is identified with $\Aut(E)$. If $\mathfrak{n}$ is the maximal ideal defining the inclusion of the residual gerbe $BG_x \hookrightarrow \MM$, we denote by $\MM_n \hookrightarrow \MM$ the $n$-th nilpotent thickening of $\MM$ at $x$ defined by $\mathfrak{n}^{n+1}$ for $n \geq 0$. By \cite[Theorem 4.16]{AlperHallRydh} there exists the coherent completion of $\MM$ at $x$, which is a complete local stack $(\widehat{\MM}_x,\widehat{x})$ and a morphism $(\widehat{\MM}_x,\widehat{x}) \to (\MM,x)$ inducing isomorphisms on the $n$-th nilpotent thickenings of $\widehat{x}$ and $x$. Moreover, since $\MM$ has a good moduli space $M$, by \cite[Theorem 4.16(3)]{AlperHallRydh} we have $\widehat{\MM}_x=\MM \times_M \Spec(\widehat{\OO}_{M,\pi(x)})$ and $\widehat{\MM}_x \to \Spec(\widehat{\OO}_{M,\pi(x)})$ is a good moduli space. Note that $\widehat{\MM}_x \to \MM$ is formally versal at $x$, i.e.\ for every commutative diagram
$$
\xymatrix{
(\widehat{\MM}_x)_0 \ar@{^{(}->}[r] & \ZZ \ar[r] \ar@{^{(}->}[d] & \widehat{\MM}_x \ar[d]\\
& \ZZ' \ar[r] \ar@{-->}[ru]& \MM
}
$$
where $\ZZ \hookrightarrow \ZZ'$ is an inclusion of local artinian stacks, there is a lift $\ZZ' \to \widehat{\MM}_x$ filling the above diagram (see \cite[Definition A.13]{AlperHallRydh}).

The next lemma is a generalization of a well-known result for moduli spaces of sheaves on a K3 surface (see \cite[Section 2.6]{Kaledin:symplectic-singularities} or \cite[Proposition 4.1(3)]{LS}). In that case the proof relies on the description of the moduli space as a GIT quotient of an open subset of a Quot scheme and on the Luna slice Theorem. In our case of moduli spaces of complexes, we apply the deep results in \cite{AlperHallRydh}, which among all imply that the stack $\MM$ is \'etale-locally a GIT quotient.
%\todo{Do we only need the following assumptions for this lemma actually? $G_x\cong$Aut$E$ is smooth reductive + $\kappa$ is $G_x$-equivariant. I think so.}
\begin{lem}
\label{lemma_bad}
Assume $E=F \oplus F$ where $F$ is $\tau$-stable of Mukai vector $v_0$. Adopt the notation of $\pi(x)$, $A$ and $\mathfrak a$ as above, then
$$\widehat{\OO}_{M,\pi(x)} \cong (\widehat{A}/\mathfrak{a})^{\Aut(E)} \cong \widehat{A}^{\Aut(E)} /(\mathfrak{a} \cap \widehat{A}^{\Aut(E)}).$$
\end{lem}
\begin{proof}
Consider the quotient stack $\TT:=[\Spec(\text{Sym}^\bullet (T_{\MM,x})) / G_x]$, where $T_{\MM,x}$ is the tangent space to $\MM$ at $x$. By definition $T_{\MM,x}= \Def_E(\C[\epsilon]) \cong \Ext^1(E,E)$, so in this case $$\TT=[\Spec{A}/\Aut(E)] \to T:=\Spec{A}\sslash\Aut(E)=\Spec{A^{\Aut(E)}},$$
which is a good moduli space. We denote by $\TT_n$ the $n$-th thickening of $\TT$ at the point $0$. As computed in the proof of \cite[Theorem 1.1]{AlperHallRydh}, since $G_x=\Aut(E)$ is linearly reductive and smooth, the isomorphisms $\MM_0=BG_x\cong\TT_0$ and $\MM_1 \cong \TT_1$ lift to closed immersions $\MM_n \hookrightarrow \TT_n$ which effectivize to a closed immersion $\widehat{\MM}_x \hookrightarrow \widehat{\TT}$, where $\widehat{\TT}:= \TT\times_T \Spec{\widehat{\OO}_{T,0}}$. Note that $\widehat{\OO}_{T,0} \cong ((A^{\Aut(E)})_{\mathfrak{m}})^{\widehat{ }} \cong (\widehat{A^{\Aut(E)}})_\mathfrak{m} \cong \widehat{A^{\Aut(E)}} \cong \widehat{A}^{\Aut(E)}$, as localization and completion with respect to a maximal ideal commute.

On the other hand, note that $G_x$ acts on the quotient $\widehat{A}/ \mathfrak{a}$. Indeed, as $\kappa$ is $G_x$-equivariant, we have $G_x (\mathfrak{a}) \subset \mathfrak{a}$, so the action on the quotient is well-defined. 

We claim that $(\widehat{A}/ \mathfrak{a})^{G_x}$ is a complete local ring. In order to prove this, we firstly show that there is an isomorphism of rings $(\widehat{A}/ \mathfrak{a})^{G_x} \cong \widehat{A}^{G_x}/ (\mathfrak{a} \cap \widehat{A}^{G_x})$. Indeed, note that the surjection $\widehat{A} \twoheadrightarrow (\widehat{A} / \mathfrak{a})$ induces surjections $\widehat{A}/ \mathfrak{m}^n \twoheadrightarrow (\widehat{A} / \mathfrak{a})/ \mathfrak{m}^n$ for every $n$. Now recall that $G_x$ is linearly reductive, so every surjection $B \twoheadrightarrow C$ of $G_x$-rings induces a surjection $B^{G_x} \twoheadrightarrow C^{G_x}$ on the invariant rings. As a consequence, we have the surjections $({\widehat{A}/ \mathfrak{m}^n})^{G_x} \twoheadrightarrow ({(\widehat{A} / \mathfrak{a})/ \mathfrak{m}^n})^{G_x}$ for every $n$. Passing to the completions, it follows that there is a surjection $\widehat{A}^{G_x} \twoheadrightarrow (\widehat{A} / \mathfrak{a})^{G_x}$. An easy computation shows that this surjection induces a surjection $\widehat{A}^{G_x}/ (\mathfrak{a} \cap \widehat{A}^{G_x}) \twoheadrightarrow (\widehat{A} / \mathfrak{a})^{G_x}$, which is injective. 

Now note that $\widehat{A}^{G_x}/ (\mathfrak{a} \cap \widehat{A}^{G_x})$ is a local ring. Indeed, $\widehat{A}^{G_x} \cong \widehat{A^{G_x}}$ is a local ring, and the quotient of a local ring is a local ring. Moreover, by \cite[Equation (4.7)]{LS} we have the explicit description of $\widehat{A}^{G_x}/ (\mathfrak{a} \cap \widehat{A}^{G_x})$ as the quotient of the ring of formal power series, i.e.\ $$\widehat{A}^{G_x}/ (\mathfrak{a} \cap \widehat{A}^{G_x}) \cong \C[\![X_1, \dots, X_4, Y_{11}, Y_{12}, \dots, Y_{44}]\!] /I,$$ 
where $I$ is an ideal of $\C[\![X_1, \dots, X_4, Y_{11}, Y_{12}, \dots, Y_{44}]\!]$ (for the precise definition see \cite[Section 4, page 10]{LS}). Since the quotient of a Noetherian complete local ring is complete, we deduce the desired properties for $(\widehat{A}/ \mathfrak{a})^{G_x}$.

Define the stack $\KK:=[\Spec{(\widehat{A}/\mathfrak{a})}/G_x]$, whose good moduli space is $\KK \to \Spec(\widehat{A}/\mathfrak{a})^{G_x}$ (see \cite[Example 8.3]{Alper}). By the above computation and \cite[Theorem 1.3]{AlperHallRydh} we have that $\KK$ is coherently complete along $x$. Let $\KK_n$ be the $n$-th thickening of $\KK$ at $0$.
The $G_x$-equivariant versal family $(\widehat{E},\widehat{\varphi})$ constructed out of $\kappa$ defines a collection of equivariant compatible objects $(E_n, \varphi_n) \in \Def_E((\widehat{A}/\mathfrak{a})/\mathfrak{m}^{n+1})$ for every $n$. Equivalently we have the compatible collection $\KK_n \to \MM$. Since $\KK$ is coherently complete, by \cite[Corollary 2.6]{AlperHallRydh} these morphisms effectivize to $\KK \to \MM$. Also $\KK$ satisfies the versality property as $(\widehat{E},\widehat{\varphi})$ does.

Now note that $\KK_0 \cong \TT_0$ and $\KK_1 \cong \TT_1$ as $\mathfrak{a} \subset \mathfrak{m}^2\widehat{A}$. Thus we have the commutative diagram
$$
\xymatrix{
\MM_1 \ar@{^{(}->}[r] \ar[d]^\cong & \widehat{\MM}_x \ar[d]\\
\KK_1  \ar@{^{(}->}[d]& \MM\\
\KK \ar@{-->}[ruu]^f \ar[ru] &
}.
$$
By the universal property of $\widehat{\MM}_x$, there exists a lifting $f: \KK \to \widehat{\MM}_x$ filling the above diagram and inducing a collection of morphisms $f_n: \KK_n \to \widehat{\MM}_x$ for every $n \geq 0$. Since $\KK_1 \cong \MM_1 \hookrightarrow \widehat{\MM}_x$ is a closed immersion, by \cite[Proposition A.8]{AlperHallRydh} we have that $f_n$ and $f$ are closed immersions. So we have the commutative diagram
$$
\xymatrix{
\KK \ar[r]^{\text{id}} \ar@{^{(}->}[d]^f & \KK \ar[d] \\
\widehat{\MM}_x \ar@{-->}[ru]^g \ar[r] & \MM .
}
$$
The versality property of $\KK$ implies that there exists a lifting $g: \widehat{\MM}_x \to \KK$ filling the above diagram. This implies that $f$ is an isomorphism. On the level of good moduli spaces, by \cite[Theorem 6.6]{Alper} this implies the isomorphism $(\widehat{A}/\mathfrak{a})^{\Aut(E)} \cong \widehat{\OO}_{M,\pi(x)}$ of local rings as we wanted.
\end{proof}

Note that since the tangent space to the moduli space at $E$ is $\Ext^1(E,E)^{\Aut(E)}$, the singular part $\Ms$ of $M$ corresponds to the locus where the dimension of the tangent space jumps, i.e.\ the locus of strictly $\tau$-semistable objects. In particular, we have a stratification
$$\Delta \subset \Ms \subset M,$$
where 
$$\Delta \cong M_\sigma(v_0) \quad \mbox{and} \quad \Ms \cong \text{Sym}^2(M_\sigma(v_0)).$$

\subsection{Affine model}
\label{subsection_affinemodel}
The affine model $Z$ for the local structure at the worst singularities of $M$ is described in \cite[Section 2]{LS}. Here we recall the definition for sake of completeness referring to \cite{LS} for the details.

Let $V$ be a $4$-dimensional $\C$-vector space with a symplectic form $\omega$. Denote by $\mathfrak{sp}(V)$ the associated symplectic Lie algebra which has dimension $10$. Consider the set $Z \subset \mathfrak{sp}(v)$ defined as
$$Z=\lbrace B \in \mathfrak{sp}(V)| B^2=0  \rbrace.$$
The ideal $I_0 \subset \C[\mathfrak{sp}(V)]$ defining $Z$ is generated by the coefficients of the matrix $B^2$, which are linearly dependent, and $Z$ has dimension $6$. The singular locus $Z_{\text{sing}}$ of $Z$  is 
$$Z_{\text{sing}}= \lbrace B \in Z | \rk(B) \leq 1 \rbrace$$
and it is defined by the ideal $L_0$ generated by the $2 \times 2$-minors of $B$. Note that $Z_{\text{sing}}$ has dimension $4$ away from the origin.

On the other hand, consider the Grassmannian $G$ parametrizing maximal isotropic subspaces $U \subset V$. Defined the incidence subvariety
$$\widetilde{Z}=\lbrace (B, U) \in Z \times  G | B(U)=0 \rbrace \subset Z \times G.$$
The canonical projection $\pi: \widetilde{Z} \to G$ to the second factor is identified with the canonical projection $T^*G \to G$. Moreover, the first projection $\sigma: \widetilde{Z} \to Z$ is a semi-small resolution. Indeed, over matrices $B \in Z_{\text{sing}}$ with $\rk(B)=1$, the fiber of $\sigma$ over $B$ is $\P((\text{ker}B / \text{im} B)^*) \cong \P^1$, while over $B=0$ the fiber is $G$.

\begin{prop}[\cite{LS}, Th\'eor\`eme 2.1]
\label{prop_theor2.1LS}
The resolution $\sigma: \tilde{Z} \to Z$ is isomorphic to the blow-up of $Z$ along $\Zs \subset Z$.
\end{prop}

A key property of $Z$ is that its singularity is rigid with respect to deformations, meaning that a deformation of $Z$ which does not change the singularities of $Z$ around the origin cannot change the singularity at the origin \cite[Th\'eor\`eme 3.1]{LS}.

\subsection{Symplectic resolution of $M$}
The main result of this section is the following.
\begin{thm}
\label{thm_symplres}
The blow-up $\tM$ of the singular locus of $M$ with the structure of reduced algebraic space is a symplectic resolution of $M$.
\end{thm}
For the definition of blow-up of an algebraic space and reduced algebraic space consult Stacks Project. The argument is due to \cite{LS} and we summarize it for the interested reader.

Theorem \ref{thm_symplres} is a consequence of the following result.

\begin{prop}[{\cite[Th\'eor\`eme 4.5]{LS}}]
\label{prop_localstruct}
Let $E:=F^{\oplus 2}$ where $F \in M_\tau(v_0)$. Then there is an isomorphism of germs of analytic spaces
$$(M,[E]) \cong (\C^4 \times Z,0).$$
\end{prop}
\begin{proof}
We use the notation introduced in Sections \ref{subsec_localstructuremoduli} and \ref{subsection_affinemodel}. By Lemma \ref{lemma_bad} we have the isomorphism $\widehat{\OO}_{M,\pi(x)} \cong \widehat{A}^{\Aut(E)} /(\mathfrak{a} \cap \widehat{A}^{\Aut(E)})$. 

Set $V=\Ext^1(F,F)$; then $\Ext^1(E,E) \cong \mathfrak{gl}_2 \otimes V \cong \mathfrak{gl}_2^{\oplus 4}$ and $\Ext^2(E,E)_0 \cong \mathfrak{sl}_2$. By \cite[Proposition 4.3]{LS}, we have $\widehat{A}^{\Aut(E)} /(\mathfrak{a} \cap \widehat{A}^{\Aut(E)}) \cong \widehat{R}/I$, where $R=\C[\C^4 \times \mathfrak{sp}(4)]$, $\widehat{R}$ is the completion of $R$ at $0$ and $I$ is an ideal of $\widehat{R}$. Moreover, by \cite[Proposition 4.3(3)]{LS} the ideal $L_0$ corresponds to the locus of strictly semistable objects via the isomorphism above and by \cite[Lemma 4.4]{LS} the ideal of initial terms of $I$ satisfies $\text{in}(I)=I_0R$. 

In order to prove $(M,[E]) \cong (\C^4 \times Z,0)$ by Artin's Theorem \cite[Corollary 1.6]{Artin} and the above observations, it is enough to show $\widehat{R}/I \cong \widehat{R} / I_0\widehat{R}$. By the computation in \cite[Section 5]{LS} the deformation of $\widehat{R}/I$ towards its normal cone is trivial. This implies the statement.
\end{proof}

\begin{proof}[Proof of Theorem \ref{thm_symplres}]
The same computation as in \cite[(2.2.4)]{OGrady} shows that the singularity of a point in $\Ms \setminus \Delta$ is of type $A_1$ transversally to $\Ms$. Thus the blow-up of $M \setminus \Delta$ in $\Ms \setminus \Delta$ is a resolution of these singularities and the symplectic form over the smooth part of $M$ extends to the exceptional divisor of this blow-up. By Propositions \ref{prop_localstruct} and \ref{prop_theor2.1LS} applied to the points in $\Delta$, we have that the blow-up $\sigma$ of $M$ in $\Ms$ is a resolution of singularities. Note that the fiber of $\sigma$ over a point in $\Delta$ is a $3$-dimensional quadric. Thus the symplectic form extends to $M$ by Hartog's Theorem.
\end{proof}

\begin{rmk}
\label{rmk_Mnormal}
Note that the moduli space $M$ is normal, as it is locally described by $Z$.
\end{rmk}

\subsection{Relative version}
In order to complete the proof of Theorem \ref{thm_OG10}, we need to apply the deep theory introduced in \cite{BLM+} about families of stability conditions and relative moduli spaces. 

Recall that given a family of cubic fourfolds $\YY \to S$ over a smooth quasi-projective variety $S$, by \cite[Lemma 30.1]{BLM+} there exists an admissible subcategory $\Ku(\YY) \subset \Db(\YY)$ which defines a family of Kuznetsov components over $S$. In particular, $\Ku(\YY_s)$ is the Kuznetsov component of $\YY_s$ for every $s \in S$. A stability condition on $\Ku(\YY)$ is a collection $\underline{\tau}=(\tau_s)_{s \in S}$ of stability conditions $\tau_s$ on $\Ku(\YY_s)$ for $s \in S$, satisfying the compatibility conditions of \cite[Definitions 20.5 and 21.15]{BLM+}.

The next result is the relative version of Theorem \ref{thm_symplres} over a $1$-dimensional base and is the generalization of \cite[Corollary 32.1]{BLM+} in the case of a non-primitive Mukai vector.

\begin{prop}
\label{prop_deformovercurve}
Let $Y$ be a cubic fourfold, let $v=2v_0$ be a Mukai vector in $H^*_{\mathrm{alg}}(\Ku(Y),\Z)$ with $(v_0,v_0)=2$ and let $\tau \in \Stab^\dag(\Ku(Y))$ be $v$-generic. Let $Y'$ be an other cubic fourfold such that there is smooth family of cubic fourfolds over a connected quasi-projective base with fibers $Y$ and $Y'$ along which $v_0$ remains a Hodge class. Then there exist a family $g: \YY \to C$ of cubic fourfolds over a connected quasi-projective curve, complex points $0, 1 \in C(\C)$ and a stability condition $\underline{\tau}$ on $\Ku(\YY)$ over $C$ such that:
\begin{enumerate}
    \item $\YY_0=Y$ and $\YY_1=Y'$.
    \item $v_0$ is a primitive vector in $H^*_{\mathrm{alg}}(\Ku(\YY_c),\Z)$ for all $c \in C$.
    \item $\tau_c$ is $v$-generic for all $c \in C$ and $\tau_0$ is a small deformation of $\tau$ so that $M_{    \tau_0}(v)=M_\tau(v)$.
    \item There exist an algebraic space $M_{\underline{\tau}}(v)$ and a proper morphism $M_{\underline{\tau}}(v) \to C$ such that every fiber is the connected component containing the singular locus $\emph{Sym}^2(M_{\tau_c}(v_0))$ of the good moduli space $M_{\tau_c}(v)$ of semistable objects in $\Ku(\YY_c)$.
    \item There exist an algebraic space $\widetilde{M_{\underline{\tau}}(v)}$ and a proper morphisms $\widetilde{M_{\underline{\tau}}(v)} \to C$ making $\widetilde{M_{\underline{\tau}}(v)}$ a relative symplectic resolution of $M_{\underline{\tau}}(v)$: its fiber over any point $c \in C$ is a symplectic resolution of the fiber of $M_{\underline{\tau}}(v)$ over $c$, obtained by blowing up the singular locus $\emph{Sym}^2(M_{\tau_c}(v_0))$. 
\end{enumerate}
\end{prop}
\begin{proof}
Properties (1)-(3) are a consequence of the assumptions and \cite[Proposition 30.7]{BLM+}. In order to prove (4), note that $M_{\tau_c}(v)$ is normal for every $c \in C$ as observed in Remark \ref{rmk_Mnormal}; in particular, $M_{\tau_c}(v)$ is locally irreducible. It follows that if $M_{\tau_c}(v)$ is reducible, then it is a finite disjoint union of normal irreducible components. Denote by $M_{\tau_c}(v)'$ the connected component containing $\text{Sym}^2(M_{\tau_c}(v_0))$. Replacing $M_{\tau_c}(v)$ with $M_{\tau_c}(v)'$, item (4) follows from \cite[Theorem 21.24(c)]{BLM+}. Finally, by \cite[Corollary 32.1]{BLM+} applied to $v_0$ we have that the algebraic space $\Sym^2(M_{\underline{\tau}}(v_0))$ is proper over $C$ and satisfies $\Sym^2(M_{\underline{\tau}}(v_0))_c=\Sym^2(M_{\tau_c}(v_0))$ for every $c \in C$. Thus we define $\widetilde{M_{\underline{\tau}}(v)}$ as the blow up of $M_{\underline{\tau}}(v)$ in $\Sym^2(M_{\underline{\tau}}(v_0))$ and we have that $\widetilde{M_{\underline{\tau}}(v)}_c$ is the blow up of $M_{\tau_c}(v)'$ in $\Sym^2(M_{\tau_c}(v_0))$. This implies (5).
\end{proof}

\subsection{Proof of Theorem \ref{thm_OG10}}
By Theorem \ref{thm_symplres}, we know $\tM$ is smooth, proper and symplectic of dimension $10$. In this paragraph, we end the proof of Theorem \ref{thm_OG10}. In particular, we show that $M$ and $\tM$ are projective, by proving that they carry an ample divisor, and that $\tM$ is deformation equivalent to the O'Grady's $10$-dimensional example.

Consider the irreducible component $M' \subset M$ containing $M^{\text{sing}}$. By abuse of notation, we still denote by $\tM$ the blow up of $M'$ in the reduced singular locus. By Proposition \ref{prop_deformovercurve}(3) we have that $M'$ is a limit of moduli spaces $M_n$ of semistable objects in the derived category of a K3 surface with Mukai vector $v$ with respect to a $v$-generic stability condition. Indeed, it is enough to choose a curve $C$ in the moduli space of cubic fourfolds, such that its intersection with the loci of cubic fourfolds having Kuznetsov component equivalent to the bounded derived category of a K3 surface is dense in $C$. By \cite[Proposition 2.2, Corollary 3.16]{MZ:OGradytype}, the moduli space $M_n$ admits a symplectic resolution $\tM_n$ which is deformation equivalent to the irreducible holomorphic symplectic manifold constructed by O’Grady in \cite{OGrady}. Then by Proposition \ref{prop_deformovercurve}(4) the blow-up $\tM$ is the limit of the smooth irreducible holomorphic symplectic varieties $\tM_n$. 

By the same argument used in \cite{BLM+}, there is a non-degenerate quadratic form $q$ defined over $H^2(\tM,\Z)$, which plays the role of the Beauville--Bogomolov--Fujiki form. By \cite[Theorem 1.14]{Perego:Kahlerness}, there is a bimeromorphic function $f: \tM \to \tM''$, where $\tM''$ is a projective irreducible holomorphic symplectic variety. Moreover, the bimeromorphic function $f$ induces an isometry $H^2(\tM,\Z) \cong H^2(\tM'',\Z)$ respecting the Beauville--Bogomolov--Fujiki forms.

Now denote by $l$ the divisor class on $M'$ constructed in \cite{BM:projectivity}. By \cite[Theorem 1.1]{BM:projectivity}, the class $l$ is strictly nef. On the other hand, if $\tilde{l}$ is the pullback via the blow-up $\sigma$ of $l$, then $q(\tilde{l})>0$. Indeed, the same statement is true for the desingularized moduli spaces of semistable objects on K3 surfaces, and the divisor class $l$ behaves well with respect to deformations by \cite[Theorem 21.25]{BLM+}. 

Let $\tilde{l}''$ be the line bundle of $\tM''$ such that $\tilde{l}=f^*\tilde{l}''$; note that $q(\tilde{l}'')=q(\tilde{l})>0$.
By \cite[Corollary 3.10]{Huybrechts:compactHyperkaehlerbasic}, $\tilde{l}''$ is big. Since $f$ is an isomorphism out of codimension 2, it follows $\tilde{l}$ is big too. Since $\tM$ has trivial canonical bundle, the Base Point Free Theorem (see \cite{KollarMori}, or \cite{Ancona:Vanishing} for algebraic spaces) implies that $m\tilde{l}$ is globally generated for a certain integer $m \gg 0$. Since by Theorem \ref{thm_symplres}, the moduli space $M$ has rational singularities, we deduce that also $ml$ is globally generated. Together with the fact that $ml$ is strictly nef, we conclude that $ml$ is ample. This implies the projectivity of $M'$, and then of $\tM$.

Finally note that since $M'$ is normal and projective, we can apply the same argument in \cite[Theorem 4.4]{Kaledin:symplectic-singularities} to deduce that $M'=M$, namely that $M$ is irreducible. The deformation type of $\tM$ is obtained by degeneration to the loci of cubic fourfolds with associated K3 surface. This ends the proof of Theorem \ref{thm_OG10}.

\section{Stable objects in the moduli space $M_\sigma(2\lambda_1+2\lambda_2)$}\label{sec:stabobjinM}
In this section, we introduce the objects which form an open subset of $M_\sigma(2\lambda_1+2\lambda_2)$. After recalling the definition of instanton sheaves on a smooth cubic threefold from \cite{Druel:Instanton}, we compute the projection of the stable ones in $\Ku(Y)$.
\begin{rmk}\label{rem:ssstable}
Comparing with $\sigma$-stable objects,  strictly semistable objects are easier to describe. Note that we may vary the stability condition $\sigma$ to $\sigma_0$ in $\Stab^\dag(\Ku(Y))$ such that
\begin{enumerate}
    \item[(a)] $M^s_\sigma(2\lambda_1+2\lambda_2)=M^s_{\sigma_0}(2\lambda_1+2\lambda_2)$ and $M^s_\sigma(\lambda_1+\lambda_2)=M^s_{\sigma_0}(\lambda_1+\lambda_2)$;
    \item[(b)] $\sigma_0$ is generic with respect with $2\lambda_1+2\lambda_2$.
\end{enumerate}
By condition (b), as the character $\lambda_1+\lambda_2$ is primitive, the Jordan--H\"{o}lder factors of  strictly $\sigma_0$-semistable objects are all with character $\lambda_1+\lambda_2$. By \cite[Theorem 1.1]{LPZ1}, such a factor is always of the form
\begin{align}\label{eq:defofPl}
  P_\ell\coloneqq  \mathsf{pr} (\OO_\ell[-1]) 
= \mathrm{Cone}(\II_\ell[-1]\xrightarrow{\mathsf{ev}} \OO_Y(-H)[1] ), 
\end{align}
where $\ell$ is a line in $Y\subset \P^5$, $\II_\ell$ denotes the ideal sheaf of $\ell$. Let $F(Y)$ be the Fano variety of lines on $Y$; then the strictly semistable locus in $M_{\sigma_0}(2\lambda_1+2\lambda_2)$ is isomorphic to $\mathrm{Sym}^2F(Y)$. %We will make further use of them in Section \ref{}.
\end{rmk}

\subsection{Moduli space of semistable instanton sheaves on a smooth cubic threefold}
\label{subsec_instanton}
Recall the definition of $\lambda_1$ and $\lambda_2$ as that in \eqref{eq:defoflambdas}. By a direct computation, their Chern characters are %\emph{Note that the projection functor can be expressed as compositions of left mutations as that in Remark \ref{rem:mutations}.}\todo{I would suggest to remove this sentence, since the fact that the functor is composition of mutations is already recalled in (3.2).} 
\begin{equation}\label{eq:choflambda}
 \ch(\lambda_1)=(3,-H,-\frac{H^2}2,\frac{H^3}6,\frac{3}8)\text{ and }\ch(\lambda_2)=  (-3,2H,0,-\frac{H^3}3,0).
\end{equation}
 In particular, we have 
 \begin{equation}\label{eq:2lambda12}
 \ch(2\lambda_1+2\lambda_2)=(0,2H,-H^2,-\frac{H^3}3,\frac{3}4).
\end{equation}
On the other hand, in \cite{Druel:Instanton} Druel studies the moduli space of semistable sheaves $F$ on a smooth cubic threefold with Chern classes $$\rk(F)=2, \; c_1(F)=0,\; c_2(F)=2 \;\text{ and }c_3(F)=0.$$ We follow the definition as that in  \cite{LMS:ACM,Kuzne:Instanton}, and call such sheaves \emph{rank $2$ instanton sheaves} on cubic threefolds. Let $X=H\cap Y$ be a smooth cubic threefold and denote by $\iota:X\rightarrow Y$ the inclusion. For such an instanton sheaf $F$, by a direct computation, we have
$$\ch(\iota_*F)=\ch(2\lambda_1+2\lambda_2).$$
We summarize the results about rank $2$ instanton sheaves in \cite{Beauville:Cubics, Druel:Instanton} as follows.
\begin{rmk}\label{rem:instanton}
Let $X$ be a smooth cubic threefold. The moduli space  $M_{\mathrm{inst}}$ of rank $2$ instanton sheaves on $X$ consists of the following objects, see \cite[Theorem 3.5]{Druel:Instanton}.
\begin{enumerate}
    \item[(i)] $F_\Gamma$: For every stable rank $2$ instanton bundle $F$, the zero locus of a non-zero section on $F(H)$ is  a non-degenerate \emph{elliptic quintic} curve $\Gamma$. Namely, an elliptic quintic curve is a \emph{locally complete intersection} quintic curve with trivial
canonical bundle and $h^0(\OO_\Gamma)=1$, and the curve is called non-degenerate if it spans $\P^4$. For a generic section, the curve $\Gamma$ is smooth.
    
    For every non-degenerate elliptic quintic curve $\Gamma$, one can produce the vector bundle $F_\Gamma$ by the Serre construction. For a more categorical description,
    \begin{equation}\label{eq:defofFGamma}
    F_\Gamma\coloneqq \mathrm{Cone}\left(\II_\Gamma(H)[-1]\xrightarrow{\mathsf{ev}}\OO_X(-H)\right).\end{equation}
    All these stable bundles form a dense affine open subset $M^{\mathrm{s}}_{\mathrm{inst}}$ in  $M_{\mathrm{inst}}$, see \cite[Corollary 6.6]{Beauville:Cubics}.
    \item[(ii)] $F_C$: Every stable non-locally-free rank $2$ instanton sheaf is one-to-one corresponding to a smooth conic curve $C$. Conversely, for each smooth conic $C$, one can define a stable sheaf $F_C$ as the kernel of 
    \begin{equation}\label{eq:defofFC}
    \OO_X\otimes \Hom(\OO_X,\theta_C(H))\xrightarrow {\mathrm{ev}}\theta_C(H).
    \end{equation}
    Here we write $\theta_C$ for the theta-characteristic of $C$ so that $\theta_C(H)\cong\OO_{\P^1}(1)$ is a degree $1$ line bundle on $X$.
    
    The locus $A$ in $M_{\mathrm{inst}}$ that parameterizes these sheaves is of dimension $4$.
    \item[(iii)] $\II_{\ell_1}\oplus \II_{\ell_2}$: Every strictly semistable rank $2$ instanton sheaves is S-equivalent to this direct sum. Here $\ell_1$ and $\ell_2$ are lines (possibly the same) on $X$.
    
    The locus $B$ in $M_{\mathrm{inst}}$ that parameterizes these sheaves is isomorphic to  $\mathrm{Sym}^2 F(X)$, where $F(X)$ stands for the Fano surface of lines on $X$.
\end{enumerate}
The following properties of $M_{\mathrm{inst}}$ are summarized from \cite[Section 4]{Druel:Instanton} and \cite[Section 6]{Beauville:Cubics}.
Let $J^2(X)$ be the the translate of the intermediate Jacobian which parameterizes $1$-cycles of degree $2$ on $X$. Consider the morphism
\begin{equation}\label{eq:chow2map}
    \mathfrak{c}_2: M_{\mathrm{inst}}\rightarrow J^2(X): F\mapsto \tilde{c}_2(F),
\end{equation} %\todo{I would suggest to use $F$ here instead of $\FF$ to be consistent with the notation used at the beginning of this section for an instanton sheaf. OK. Then we change the fano surfaces to mathfrak F. We then change the F's in Sec 6 simultaneously. It could be a heavy notation and weird to read since we use $J$ for the Jacobian. I would suggest to leave $F$ for the Fano surface. Here there is $F(X)$ and $F_{\mathrm{conic}}(X)$, so there shouldn't be conflict with $F$ of the instanton sheaves. An other possibility is to change the instanton sheaf, maybe denote it by $G$?}
where $\tilde{c}_2(F)$ corresponds to the Chern class $c_2(F)$ via the cycle map.
\begin{enumerate}
    \item The moduli space $M_{\mathrm{inst}}$ is smooth and connected. The morphism $\mathfrak c_2$ induces an isorphism of $M^{\mathrm{s}}_{\mathrm{inst}}$ onto its image in $J^2(X)$.
    \item The morphism $\mathfrak c_2$ contracts the locus $A$ to  $F_{\mathrm{conic}}(X)\subset J^2(X)$, where $ F_{\mathrm{conic}}(X)$ is the image of the variety of conics and is isomorphic to $F(X)$. The morphism $\mathfrak c_2$ is isomorphic to the blowing up of $J^2(X)$ along $F_{\mathrm{conic}}(X)$.
    \item The morphism $\mathfrak c_2$ maps $B$ onto an ample divisor $D_{ F+F}$ in $J^2(X)$.
\end{enumerate}  
We will make use of these further details on $M_{\mathrm{inst}}$ in Section \ref{sec:Lagrangian}.
\end{rmk}

\subsection{Formulas of $E_\Gamma$ and $E_C$}
The classification of semistable rank $2$ instanton sheaves summarized in the previous section inspires us the construction of some objects in $\Ku(Y)$ with character $2\lambda_1+2\lambda_2$.  Recall the definition of the projection functor $\mathsf{pr}=\mathsf R_{\OO_Y(-H)}\mathsf R_{\OO_Y(-2H)}\mathsf L_{\OO_Y}=\mathsf L_{\OO_Y}\mathsf R_{\OO_Y(-H)}\mathsf R_{\OO_Y(-2H)}$ as in \eqref{eq:defofpr}.
\begin{defn}\label{def:EGammaEC}
Let $Y$ be a smooth cubic fourfold. Let $\Gamma$ be an elliptic quintic curve on $Y$. We define the object $E_\Gamma$ as:
\begin{equation}\label{eq:defofEGamma} 
    E_\Gamma\coloneqq \mathsf{pr}(\II_{\Gamma}(H)).
\end{equation}
Let $C$ be a smooth conic curve on $Y$. We define the object $E_C$ as:
\begin{equation}\label{eq:defofEC} 
    E_C\coloneqq \mathsf {pr}(\theta_C(H))[-1].
\end{equation}
\end{defn}
 We will need a more explicit expression of  $E_C$, as computed in the following lemma.
\begin{lem}\label{lem:formofEC}
The object $E_C$ can be written as
\begin{equation}\label{eq:formularofEC}
    \mathrm{Cone}\left(\mathsf L_{\OO_Y}(\theta_C(H))[-2]\xrightarrow{\mathsf{ev}}\OO_Y(-H)[1]\otimes (\Hom(\mathsf L_{\OO_Y}(\theta_C(H)),\OO_Y(-H)[3]))^*\right).
\end{equation}
\end{lem}
\begin{proof}
 Note that
\begin{equation}\label{eq:homofOC}
    \Hom_{\Db(Y)}(\OO_Y,\theta_C(H)[i])=\Hom_{\OO_C}(\OO_C,\theta_C(H)[i]) = \begin{cases}
\C^2, & \text{when }i=0;\\ 0, & \text{when }i\neq 0.
\end{cases}
\end{equation}
In particular, the object $\mathsf L_{\OO_Y} (\theta_C(H))[-1]$ is a coherent sheaf on $Y$. %and its character is given by  $$\ch(\mathsf L_{\OO_Y} \OO_C(1)[-1])=2\ch(\OO_Y)-\ch(\OO_C(1)).$$
Note that $$\theta_C(H),\OO_Y\in \OO_Y(H)^\perp= \ ^\perp\! \OO_Y(-2H),$$ therefore the object $\mathsf L_{\OO_Y} (\theta_C(H))$ is also in $^\perp\! \OO_Y(-2H)$, in other words, $$\mathsf R_{\OO_Y(-2H)}\mathsf L_{\OO_Y}(\theta_C(H))=\mathsf L_{\OO_Y}(\theta_C(H)).$$   

Since $\OO_Y\in\;^\perp\! \OO_Y(-H)$, by Serre duality we have 
\begin{align}
 & \Hom_{\Db(Y)}(\mathsf L_{\OO_Y}(\theta_C(H)),\OO_Y(-H)[i])\cong \Hom_{\Db(Y)}(\theta_C(H),\OO_Y(-H)[i])    \\ 
 \cong  & (\Hom_{\Db(Y)}(\OO_Y(-H),\theta_C(-2H)[4-i]))^*=\begin{cases}
 \C^2, & \text{when }i=3;\\
 0, & \text{when }i\neq 3.
 \end{cases}\label{eq:homofOYOC}
\end{align}
By Definition \ref{def:mutationfunctor}, the formula \eqref{eq:formularofEC} for $E_C$ holds.
\end{proof}

\begin{prop}\label{prop:EGamaECinKu}
Let $X$ be a smooth cubic threefold on $Y$ and $\iota:X\rightarrow Y$ be the embedding morphism. We have the following statements for objects of the form $E_C$ and $E_\Gamma$.
\begin{enumerate}
    \item If $C$ is a smooth conic contained in $X$, then $$E_C\cong \mathsf{pr} (\iota_*F_C),$$ where $F_C$ is as that defined in \eqref{eq:defofFC}.
    \item If $\Gamma$ is a non-degenerate elliptic quintic curve contained in $X$, then $E_\Gamma\cong\iota_*F_\Gamma$ as that defined in \eqref{eq:defofFGamma}. In particular, the object $E_\Gamma$ sits in the short exact sequence in $\Coh(Y)$: \begin{equation}\label{eq:trianofEGamma}
   0\to\OO_X(-H)\to E_\Gamma\rightarrow \II_{\Gamma/X}(H)\to 0.
\end{equation}
    \item Let $\ell$ be a line on $X$; then $P_\ell\cong \mathsf{pr}(\II_{\ell/X})$.
    \item Both $E_\Gamma$ and $E_C$ are in $\Ku(Y)$ with character $2\lambda_1+2\lambda_2$.
\end{enumerate}
 
\end{prop}
\begin{proof}(1). When $C$ is contained in a smooth cubic threefold $X$, note that $F_C$ is stable on $X$, so we have $\Hom(\OO_Y,\iota_*F_C)=0$. Note that $\iota_*F_C=\mathrm{Cone}(\iota_*\OO_X^{\oplus 2}\rightarrow \theta_C(H))[-1]$ and $\Hom(\OO_Y,\iota_*\OO_X[i])=0$ when $i\neq 0$. Together with \eqref{eq:homofOC}, this implies that $\iota_*F_C\in \OO_Y^\perp$.
Since both $\OO_C,\OO_X\in \OO_Y(H)^\perp$, we have $$\mathsf{pr}(\iota_*F_C)=\mathsf R_{\OO_Y(-H)}(\iota_*F_C).$$
By Serre duality
\begin{equation}\label{eq:homofOX}
  \Hom_{\Db(Y)}(\iota_*\OO_X,\OO_Y(-H)[i])=\Hom_{\OO_X}(\OO_X,\OO_X(-2H)[4-i])= \begin{cases}
\C, & \text{when }i=1;\\ 0, & \text{when }i\neq 1.\end{cases}
\end{equation}
The unique extension gives the obvious triangle $\iota_*\OO_X\rightarrow \OO_Y(-H)[1]\rightarrow \OO_Y[1]\xrightarrow{+}$. Therefore, we have the natural commutative diagram of distinguished triangles:
$$
\xymatrix{
\iota_*\OO_X^{\oplus 2}[-1] \ar[r] \ar[d]^{\mathsf{ev}} & \OO_Y(-H)^{\oplus 2} \ar[r] \ar[d] & \OO_Y^{\oplus 2} \ar[d]^{\mathsf{ev}} \ar[r]^+& \\
\theta_C(H)[-1] \ar[r] \ar[d] &  0 \ar[r] \ar[d] & \theta_C(H) \ar[d] \ar[r]^+& \\
\iota_*F_C \ar[r]^{\mathsf{ev}} & \OO_Y(-H)[1]^{\oplus 2} \ar[r] & \mathsf{L}_{\OO_Y}(\theta_C(H)) \ar[r]^{\;\;+}& .
}
$$
The morphism $\mathsf{ev}$ at the bottom line is $\iota_*F_C\rightarrow \OO_Y(-H)[1]\otimes (\Hom(\iota_*F_C,\OO_Y(-H)[1]))^*$. By \eqref{eq:homofOYOC} and \eqref{eq:homofOX}, the object $\mathsf R_{\OO_Y(-H)}(\iota_*F_C)$ is 
\begin{equation*}
    \mathrm{Cone}\left(\iota_*F_C\xrightarrow{\mathsf{ev}}\OO_Y(-H)^{\oplus 2}[1]\bigoplus \OO_Y(-H)[2]\otimes (\Hom(\mathsf L_{\OO_Y}(\theta_C(H)),\OO_Y(-H)[2]))^*\right)[-1],
\end{equation*}
which is isomorphic to the object as that in \eqref{eq:formularofEC}.

(2). Consider the short exact sequence
$$0 \to \II_{\Gamma}(H) \to \OO_Y(H) \to \OO_{\Gamma}(H) \to 0.$$
As $h^1(\OO_{\Gamma}(H))=h^0(\OO_{\Gamma}(-H))=0$ and $\chi(\OO_{\Gamma}(mH))=5m$, we have $h^0(\OO_{\Gamma}(H))=5$. Note that $h^0(\OO_Y(H))=6$ and the induced map
$$H^0(\OO_Y(H)) \to H^0(\OO_{\Gamma}(H))$$
is surjective, since the linear span of $\Gamma$ is a $\P^4$. We conclude that 
$$\mathsf{L}_{\OO_Y}(\II_\Gamma(H))=\mathrm{Cone}\left(\OO_Y \xrightarrow{\mathsf{ev}} \II_{\Gamma/Y}(H) \right)=\II_{\Gamma/X}(H).$$
Consider on the category $\OO_Y^\perp$ the following semiorthogonal decomposition with two components:
\begin{equation}\label{eq:sodfokuy1}
    \langle \langle \OO_Y(-2H),\OO_Y(-H)\rangle, \Ku(Y)\rangle.
\end{equation}
Consider the expression of $\II_{\Gamma/X}(H)$ as that in Definition \ref{def:sod} (b):
\begin{equation}
    0=F_0\rightarrow F_1=\iota_*F_\Gamma\xrightarrow{\mathsf{ev}} F_2=\II_{\Gamma/X}(H).
\end{equation}
  Here Cone$(F_1\to F_2)$ is $\OO_X(-H)[1]$ by \eqref{eq:defofFGamma}. Note that $$\OO_X(-H)[1]=\mathrm{Cone}(\OO_Y(-2H)\to\OO_Y(-H))[1]\in \langle \OO_Y(-2H),\OO_Y(-H)\rangle.$$
  By \cite[Lemma 3.1]{Kuzne:Instanton}, \begin{equation}
      \Hom_{\OO_Y}(\OO_Y(jH),\iota_*F_\Gamma[i])=\Hom_{\OO_X}(\OO_X(jH),F_\Gamma[i])=0
  \end{equation}
for $j=0,1,2$ and all $i\in \Z$. Therefore, $\iota_*F_\Gamma\in \Ku(Y)$. By Remark \ref{rem:mutations}, the functor $\mathsf{pr}_{\Ku(Y)}$ with respect to \eqref{eq:sodfokuy1} maps $\mathsf{L}_{\OO_Y}(\II_\Gamma(H))$ to $\iota_*F_\Gamma$. 

(3). Since $\II_{\ell/X}\in \langle \OO_Y,\OO_Y(H)\rangle^\perp$, we have $\mathsf{pr}(\II_{\ell/X})=\mathsf{R}_{\OO_Y(-H)}(\II_{\ell/X})$. By the same argument as that for the conic case, the statement holds.

(4). By \eqref{eq:homofOC} and \eqref{eq:formularofEC}, the character of $E_C$ is
\begin{align*}
    \ch(E_C)&= \ch(\mathsf L_{\OO_Y}(\theta_C(H))[-1])-2\ch(\OO_Y(-H))\\ &=2\ch(\OO_Y)-\ch(\theta_C(H))-2\ch(\OO_Y(-H))=\ch(2\lambda_1+\lambda_2).
\end{align*}
The Chern character of $E_\Gamma$ is  $\ch(\iota_*F_\Gamma)$ which is the same as $\ch(2\lambda_1+2\lambda_2)$.
\end{proof}

\section{Stability of $E_\Gamma$ and $E_C$} 
\label{sec:stabE}
In this technical section, we prove Theorem \ref{thm:EgammaECinkustab}, namely, Theorem \ref{thm:EgammaECinkustabintro} in the introduction. In particular, we study the essential image of the objects $E_\Gamma$ and $E_C$ defined in Section \ref{sec:stabobjinM} via the equivalence between $\Ku(Y)$ and $\Ku(\P^3, \BB_0)$ of Proposition \ref{prop:equivalentofKu}. We show that these objects are stable with respect to tilt-stability $\sigma_{\alpha, -1}$ on $\Db(\P^3,\BB_0)$ by a wall-crossing argument.

For this purpose, we will inevitably work with details about the category $\Db(\P^3,\BB_0)$ and we will prove some additional properties used for the computation, which are also of independent interest. This is the only section where $\Db(\P^3,\BB_0)$ is involved. For readers not familiar with this setting, there is no harm to skip the whole section, since the only result that we will use in the rest of the paper is Theorem \ref{thm:EgammaECinkustab}. %\todo{I have added some sentences at the beginning of each subsection. Feel free to delete if you don't like!}
\subsection{More on the equivalence between $\Ku(Y)$ and $\Ku(\P^3,\BB_0)$} Recall that the construction of the stability condition $\sigma$ on $\Ku(Y)$ is via pull-back of the stability condition induced on $\Ku(\P^3,\BB_0)$. In particular, in order to prove the stability of an object $E$ in $\Ku(Y)$, we need to show that $\Psi(\rho^* E)$ in $\Ku(\P^3,\BB_0)$ is stable. In this section, we recall some properties of the functor $\Psi$ which we will use in the next.

Recall from \eqref{diag:blowupL} and Remark \ref{rem:equivalentofKu} that $\rho:\tilde{Y}\rightarrow Y$ is the blow-up morphism. The functor $\Psi$ is defined as that in \cite[(6)]{BLMS:kuzcomponent} by
\begin{equation}\label{eq:defofPsi}
    \Psi(-)=\pi_*(- \otimes  \EE[1]):\Db(\tilde Y)\rightarrow \Db(\P^3,\BB_0).
\end{equation}
Here $\EE$ is a sheaf of right $\pi^*\BB_0$-modules on $\tilde Y$ defined by the short exact sequence of right $q^*\BB_0$-modules
\begin{equation}\label{eq:defofcE}
    0\to \OO_{\P_{\P^3} (\FF)/\P^3}(-2)\otimes q^*\BB_{1}\xrightarrow{\delta_{-1,2}} \OO_{\P_{\P^3} (\FF)/\P^3}(-1)\otimes q^*\BB_2\to \alpha_*\EE\to 0,
\end{equation}
where the morphism $\delta_{-1,2}$ is defined in \cite[Section 3.1 and 3.4]{Kuz:Quadric}. We would not use further details about $\delta_{-1,2}$ here, but the following fact  will be important for us.
\begin{lem}[{\cite[Lemma 4.7]{Kuz:Quadric}}]\label{lem:Eisvb}
 The $\OO_{\tilde Y}$-coherent sheaf $\Forg(\EE)$ is locally free with rank $2$.
\end{lem}
By \cite[Lemma 4.12]{Kuz:Quadric} and \cite[Proposition 7.7]{BLMS:kuzcomponent}, the image of some objects under $\Psi\circ\rho^*$ is as follows:
\begin{equation}\label{eq:imageofOHs}
\Psi(\rho^*\OO_Y)=0;\;\;\Psi(\rho^*\OO_Y(-H))=\BB_{-1};\;\;\Psi(\rho^*\OO_Y(H))=\BB_{2}[1].
\end{equation}
 By \cite[Lemma 4.10]{Kuz:Quadric}, the functor $\Psi$ has a left adjoint functor 
\begin{equation}\label{eq:defofphi}
    \Phi(-)=\pi^*(-)\otimes_{\pi^*\BB_0}\EE',
\end{equation} 
where $\EE'$ is a left $\pi^*\BB_0$-module  on $\tilde Y$ defined by the following short exact sequence of left $q^*\BB_0$-modules:
\begin{equation}\label{eq:defofcEpie}
    0\to \OO_{\P_{\P^3} (\FF)/\P^3}(-2)\otimes q^*\BB_{0}\xrightarrow{\delta'_{-1,1}} \OO_{\P_{\P^3} (\FF)/\P^3}(-1)\otimes q^*\BB_1\to \alpha_*\EE'\to 0.
\end{equation}
\begin{rmk}\label{rem:charofpsi2lambda}
The rank of torsion-free $\BB_0$-modules on $\P^3$ is always a multiple of $4$ by \cite[Remark 8.4]{BLMS:kuzcomponent}. The functor $\Psi\rho^*$ maps the characters $\lambda_1$ and  $\lambda_2$ to the twisted Chern characters
\begin{equation}\label{eq:charoflambdas}
    \ch_{\BB_0,\leq 2}^{-1}(\Psi\rho^*(\lambda_1))=(4,3,-\frac{7}8)\text{ and }\ch_{\BB_0,\leq 2}^{-1}(\Psi\rho^*(\lambda_2))=(-8,0,\frac{7}4)
\end{equation}respectively. 
In particular, for an object $E$ in $\Ku(Y)$ with   character $2\lambda_1+2\lambda_2$, the twisted Chern character of $\Psi\rho^*(E)$ is\begin{equation}\label{eq:char-86}
  \ch_{\BB_0,\leq 2}^{-1}(E)=(-8,6,\frac{7}4).
\end{equation}
\end{rmk}

\subsection{Expression of $\Psi\rho^*(E_{\Gamma})$} \label{subsec_expressionGamma}
Let  $\Gamma$ be a non-degenerate smooth elliptic quintic contained in a smooth cubic threefold $X$ (which is unique). By the formula \eqref{eq:trianofEGamma}, the object $\Psi\rho^*(E_{\Gamma})$ sits in the distinguished triangle 
\begin{equation}\label{eq:triangleofpsiegamma}
    \Psi\rho^*(\OO_X(-H)) \to \Psi\rho^*(E_{\Gamma}) \to \Psi\rho^*(\II_{\Gamma/X}(H))\xrightarrow{+}.
\end{equation}

Recall that the morphism $\rho$ is the blow-up along a line $L$ on $Y$, and by Proposition \ref{prop:change_line}, the choice of $L$  does not affect the stability of an object in $\Ku(Y)$. As a consequence, we can choose $L$ such that $\Psi\rho^*(E_{\Gamma})$ has a more explicit and nicer description. More precisely, given $\Gamma$ and $X$, we may choose the line $L$ not contained in a plane on $Y$ such that:\begin{cond}\label{cond:choiceofL}
\begin{enumerate}
    \item[(a)] the line $L$ intersects $X$ at a point $P$; 
    \item[(b)] the point $P$ is not on the segment variety of $\Gamma$ (since $X$ is smooth, the segment variety of $\Gamma$ does not contain $X$);
    \item[(c)] the point $P$ is only on finitely many lines on $X$.
\end{enumerate}
\end{cond}
By condition (a), the restriction of $\rho$ to $\rho^{-1}(X)$ is the blow-up  $\tilde{X}$ of $X$ in the point $P$. By condition (c), a plane  containing $L$ intersects with $X$ at either three points (counting multiplicity) including $P$ or a line through $P$. A fiber of $\pi|_{\tilde X}:\tilde X\to \P^3$ is either two points or a line. 

%In particular, the fibers of $\pi$ in $\tilde{X}$ are either two points or lines in $X$ passing through $P$. Indeed, if $\P(V_6) \cong \P^5$, $L=\P(V_2)$ and $P=[\bar{v}]$ for $\bar{v} \in V_2$, then
%$$q^{-1}([v])=\P(\langle v,\bar{v}\rangle)\cong \P^1.$$
%If this line is contained in $X$, then
%$$\pi^{-1}([v])=\P(\langle v,\bar{v}\rangle),$$
%as we claimed.
%\todo{here $X$ is assumed to be smooth. not 100 percents sure why: more assumptions on $\Gamma$? Here it seems to me that it is equivalent to assume that $Y$ does not contain any segment variety of any elliptic quintic.}
By condition (b), the image $\pi\rho^{-1}(\Gamma)$ in $\P^3$ is isomorphic to $\Gamma$. By definition \eqref{eq:defofPsi} and Lemma \ref{lem:Eisvb},
\begin{equation}\label{eq:TGamma}
    \Psi \rho^*(\OO_{\Gamma}(H))=\TT_{\Gamma}[1]
\end{equation}
for some  torsion sheaf $\TT_{\Gamma}$  supported on $\pi\rho^{-1}(\Gamma)$. By \eqref{eq:imageofOHs}, we have
\begin{equation}\label{eq:imageofOx}
    \Psi\rho^*(\OO_X(H))=\Psi\rho^*(\OO_Y(H))=\BB_2[1].
\end{equation}
We deduce the distinguished triangle for one object in \eqref{eq:triangleofpsiegamma}: 
\begin{equation}
\label{eq_PsiIGamma}
 \TT_{\Gamma}\to \Psi\rho^*(\II_{\Gamma/X}(H)) \to \BB_2[1]\xrightarrow{+} . 
\end{equation}

In order to compute the other factor $\Psi\rho^*(\OO_X(-H))$ in \eqref{eq:triangleofpsiegamma}, we consider the sequence
$$0 \to \OO_X(-H) \to \OO_X \to \OO_S \to 0,$$
where $S$ is a smooth cubic surface not containing $P$. The  object
$$\Psi \rho^*\OO_S=\TT_S[1],$$
where $\TT_S$ is a torsion $\BB_0$-module supported on $\pi\rho^{-1}(S)$. On the other hand, by \eqref{eq:imageofOHs}, we have
$$\Psi \rho^*(\OO_X)=\Psi \rho^*(\OO_Y(-H)[1])=\BB_{-1}[1].$$
In conclusion, we have the distinguished triangle
\begin{equation}\label{eq:formofpsiox-h}
\TT_S\to \Psi \rho^*(\OO_X(-H)) \to \BB_{-1}[1]\xrightarrow{+}.
\end{equation}

Putting everything together, we observe the following property of $\Psi\rho^*(E_\Gamma)$.
\begin{lem}\label{lem:homvanforEgamma}
Let $\Gamma$ be a non-degenerate smooth elliptic quintic spanning  a smooth cubic threefold $X$. Then \begin{equation}\label{eq:homvanofEgamma}
    \Hom(\BB_i[1],\Psi\rho^*(E_\Gamma))=0
\end{equation} 
for every $i\geq 1$.
\end{lem}
\begin{proof}
By definition, the object $E_\Gamma$ is in $\Ku(Y)$, so the object  $\Psi\rho^*(E_\Gamma))$ is in $\Ku(\P^3,\BB_0)$. Therefore \eqref{eq:homvanofEgamma} holds for $i=1,2,3$. 

When $i\geq 4$, we may apply $\Hom(\BB_i[1],-)$ to the triangles \eqref{eq:triangleofpsiegamma}, \eqref{eq_PsiIGamma} and \eqref{eq:formofpsiox-h}. The vanishing holds for every factor, therefore \eqref{eq:homvanofEgamma} holds.
\end{proof}
\begin{comment}
we have the diagram
\begin{equation}
\label{cd_ext}
\xymatrix{
\Psi\sigma^*(\OO_X(-H)) \ar[r]\ar[d] & \Psi\sigma^*E_{\Gamma} \ar[r]  & \Psi\sigma^*\II_{\Gamma/X}(H)  \ar[d]  \\
\BB_{-1}[1] \ar[d] &   & \BB_2[1] \ar[d]\\
\TT_S[1]  &  & \TT_{\Gamma}[1],
}
\end{equation}
and the characters $\ch_{\BB_0,\leq 2}^{-1}$ of these objects are
\begin{equation*}
\xymatrix{
(-4,9,-57/8) \ar[r]\ar[d] & (-8,6,7/4) \ar[r]  & (-4,-3,71/8)  \ar[d]  \\
(-4,3,-9/8)\ar[d] &   & (-4,-3,-9/8)\ar[d] \\
(0,-6,6) & & (0,0,-10).
}
\end{equation*}
\begin{rmk}
By \cite[Remark 9.4]{BLMS:kuzcomponent}, we have $\BB_j$ (resp. $\BB_j[1]$) in $\Coh^{\beta}(\P^3,\BB_0)$ if and only if $\beta > \frac{-5+2j}{4}$ (resp.\ $\beta \leq \frac{-5+2j}{4}$). Thus for $\beta=-1$, %the object $\Psi\sigma^*(\OO_Y(-2H))[1]$ is in the heart. In particular, also  $\Psi\sigma^*(\OO_X(-H))$ is in the heart. Moreover, 
as $\BB_2$ and $T_{\Gamma}$ are in $\TT^{-1}$, we know that $\Psi\sigma^*\II_{\Gamma/X}(H)[-1]$ is in $\TT^{-1}$. Thus $\Psi\sigma^*\II_{\Gamma/X}(H) \in \Coh^{-1}(\P^3,\BB_0)[1]$.
\end{rmk}
\end{comment}

\subsection{Potential destabilizing objects for $\Psi(\rho^*E_\Gamma)$ and $\Psi(\rho^*E_C)$ in $\Coh^{-1}(\P^3,\BB_0)$} 
In this section, we prove some lemmas which will be useful to characterize the potential destabilizing objects of $\Psi(\rho^*E_\Gamma)$ and $\Psi(\rho^*E_C)$. In order to do this, we need the following natural definition.

\begin{defn}
Let  $F$ be an object in $\Coh(\P^3,\BB_0)$, we define $$F^{*}\coloneqq\mathcal{H}om_{\OO_{\P^3}}(F,\OO_{\P^3})$$ as the dual of $F$. Note that $\OO_{\P^3}$ is  the center of the algebra $\BB_0$. The dual sheaf $F^*$ becomes a left $\BB_0$-module. The double dual $F^{**}$ is a right $\BB_0$-module. When $F$ is a torsion-free $\OO_{\P^3}$-module, its double dual $F^{**}$ is  reflexive as a $\OO_{\P^3}$-module and we have the natural inclusion 
$$F\hookrightarrow F^{**} \rightarrow F_{\mathrm{s}}$$
as a right $\BB_0$-module. Here $F_{\mathrm{s}}$ is a torsion $\BB_0$-module and $\dim\mathrm{supp}(F_{\mathrm{s}})\leq 1$. 
\label{def:doubledual}
\end{defn}

Recall that the tilt-stability condition $\sigma_{\alpha,\beta}$ is defined in Proposition \ref{prop:constructstabcond}.

\begin{lem}[{\cite[Lemma 3.2]{LPZ1}}]\label{lem:dlt0obj}
Let $E$ be a $\sigma_{\alpha_0,\beta_0}$-semistable object in $\Coh^{\beta_0}(\P^3,\BB_0)$ for some $\alpha_0 > 0$ and $\beta_0 \in \R$. Assume that $\dlt(E)= 0$ and $\rk(E)<0$. Then
$$E=\BB_i^{\oplus n}[1] \quad \text{ for some } i\in \Z \text{ and } n\in\N.$$
\end{lem}

\begin{lem}
Let $F\in\Coh(\P^3,\BB_0)$ be reflexive as an $\OO_{\P^3}$-module with rank $4$, then $F\cong \BB_i$ for some $i\in\Z$.
\label{lem:rk4b0alg}
\end{lem}
\begin{proof}
As $F$ is a reflexive $\OO_{\P^3}$-module, we may choose a general hyperplane section $\P^2\cong H\subset \P^3$ such that the restricted sheaf $F|_H$ is a $\BB_0|_H$-module and locally free as a $\OO_H$-module. Since $F|_H$ is of rank $4$ and torsion free, it is a slope stable $\BB_0|_H$-module in $\Coh(H,\BB_0|_H)$. By \cite[Proposition 2.12]{BMMS:Cubics}, the numerical character $$\mathrm{ch}(F|_H)=\mathrm{ch}(\BB_i|_H)+(0,0,-m)$$ for some $i, m\in\Z$. By \cite[Remark 2.2 and Lemma 2.4]{LMS:ACM}, we have 
$$1\geq \chi_{\BB_0|_H}(F|_H,F|_H)= \chi(\BB_i|_H,\BB_i|_H)-2m=1-2m.$$

Hence, $m\geq 0$. We denote $\mathfrak M$ as the moduli space of semistable $\BB_0|_H$-modules with
numerical class $[F|_H]$. By the same argument as that for \cite[Theorem 2.12]{LMS:ACM}, this moduli space is irreducible and smooth with dimension $2m$. A generic point in $\mathfrak M$ stands for the object $\Ker(\BB_i\twoheadrightarrow \OO_Z)$, where $Z$ is a $0$-dimensional subscheme with length $m$ on $H$. By the semi-continuity property, we have
$$\Hom_{\BB_0|_H}(F|_H,\BB_i|_H)\neq 0.$$

As $F|_H$ is locally free with the smallest possible rank as a $\BB_0$-module, we have $F|_H\cong \BB_i|_H$. In particular,
we have $\mathrm{ch}(F|_H)=\mathrm{ch}(\BB_i|_H)$, which implies
\begin{equation}
    \ch_{\BB_0,\leq 2}^{-1}(F|_H)=
\ch_{\BB_0,\leq 2}^{-1}(\BB_i) \implies \dlt(F)= 0.
\label{eq:delta=0}
\end{equation}

For any $T\in\Coh(\P^3,\BB_0)$ such that $\dim\mathrm{supp}(T)\leq 1$, we have Hom$_{\BB_0}(T,F[1])=0$ as otherwise,  by \cite[Lemma 2.15]{BMMS:Cubics} (same statement holds for $(\P^3,\BB_0)$-algebra),  $F$ is strictly contained in another torsion-free $\BB_0$-module $F'$ with the same rank and degree. This contradicts the assumption that $F$ is a reflexive $\OO_{\P^3}$-module.

The object $F[1]$ is therefore $\sigma_{\alpha,\beta}$-stable for $\alpha \gg 0$ and $\beta>\mu_{\mathrm{slope}}(\BB_i)$. By (\ref{eq:delta=0}) and Lemma \ref{lem:dlt0obj}, we have $F\cong\BB_i$.
\end{proof}
\textbf{Notation:}
For an object $F$ in $\Db(\P^3,\BB_0)$, we denote $\HH^{i}(F)\coloneqq \HH^{i}_{\Coh(\P^3,\BB_0)}(F)$ for $i\in \Z$.
\begin{cor}
Let $F$ be an object in $\Coh^{-1}(\P^3,\BB_0)$ with rank $-4$ such that $F$ is  $\sigma_{\alpha,-1}$-stable for $\alpha \gg 0$. Then $\HH^{-1}(F)$ is $\BB_i$, and $\HH^{0}(F)$ is either $0$ or a torsion sheaf such that $\dim\mathrm{supp}(\HH^{0}(F))=1$.
\label{cor:rk4stontheright}
\end{cor}
\begin{proof}
%\emph{By the definition of $\Coh^{-1}(\P^3,\BB_0)$, the object $F$ can be written as a two-term complex. If $\dim\mathrm{supp}(\HH^{0}(F))\geq2$, then $F$ is codestabilized by $\HH^{0}(F)$. Therefore, the sheaf $\HH^{0}(F)$ is torsion or $0$, and the sheaf $\HH^{-1}(F)$ is torsion-free of rank $4$.} \todo{I would suggest to change the above sentences with what follows.} \laura{
By \cite[Lemma 2.7(c)]{BMS:stabCY3s}, the sheaf $\HH^{-1}(F)$ is torsion-free of rank $4$, and the sheaf $\HH^{0}(F)$ is either $0$ or torsion supported in dimension $\leq 1$.
Consider the double dual of $\HH^{-1}(F)$:
\begin{equation}
0\rightarrow \HH^{-1}(F)\rightarrow (\HH^{-1}(F))^{**}\rightarrow F_s\rightarrow 0.
    \label{eq:extofdd}
\end{equation}
If $F_s$ is non-zero, then we have $\mu_{\alpha,-1}(F_s)=+\infty$ and the injective map 
$$0\rightarrow \Hom_{\BB_0}(F_s,\HH^{-1}(F)[1])\rightarrow \Hom_{\BB_0}(F_s,F). $$
As (\ref{eq:extofdd}) is non-split, we have $\Hom_{\BB_0}(F_s,\HH^{-1}(F)[1])\neq 0$. In particular, $\Hom_{\BB_0}(F_s,F)\neq 0$. This contradicts the stability of $F$. Therefore,  the sheaf $\HH^{-1}(F)$ is reflexive as an $\OO_{\P^3}$-module of rank $4$. By Lemma \ref{lem:rk4b0alg},  $\HH^{-1}(F)\cong \BB_i$ for some $i\leq 0$.

If $\HH^0(F)$ is non-zero, we must have $\Hom_{\BB_0}(\HH^0(F),\HH^{-1}(F)[2])\neq 0$, since otherwise $F=\HH^0(F)\oplus \HH^{-1}(F)[1]$. Note that $\HH^{-1}(F)=\BB_i$ is locally free, so the dimension of the support of $\HH^0(F)$ must be $1$.
\end{proof}

\subsection{Tilt-stability of $\Psi(\rho^*E_{\Gamma})$} \label{subsec_tiltstablGamma}
We are now ready to show the stability of $\Psi(\rho^*E_{\Gamma})$ with respect to $\sigma_{\alpha, -1}$ for $\alpha$ large enough. The following basic commutative algebra lemma will be useful.
\begin{lem}
Let $V$ be a smooth variety and $U$ be a smooth subvariety with dimension $n$. Denote the embedding map by $\iota: U \rightarrow V$. Let $\mathcal G$ be a locally free sheaf on $U$, and $\mathcal F$ be a coherent sheaf on $V$ such that 
$\dim(\mathrm{supp}(\mathcal F)\cap U)=l.$
Then we have $$\mathrm{Ext}^i_{\mathcal O_V}(\mathcal F,\iota_*\mathcal G)=0$$
for $i<n-l$.
\label{lem:caextvanish}
\end{lem}

\begin{proof}
Let $m$ be the dimension of $V$, by Serre duality, we need to show that $$\mathrm{Ext}^i_{\mathcal O_V}(\iota_*\mathcal G,\FF)=0$$
for $i>m-n+l$ and every $\FF$ as that in the statement.
By the local to global spectral sequence, we have
$$E^{p,q}_2=H^p(\EE xt^q_{\mathcal O_V}(\iota_*\mathcal G,\FF))\Rightarrow \mathrm{Ext}^{p+q}_{\mathcal O_Y}(\iota_*\mathcal G,\FF).$$
Since $\dim(\mathrm{supp}(\mathcal F)\cap U)=l$, we have $E^{p,q}_2=0$ when $p>l$.

For any closed point $x\in U$, since $\GG$ is locally free on $U$, we have $\iota_*\GG_x\cong \OO_{U,x}^{\oplus r}$ as an $\OO_{V,x}$-module. Since $U$ is smooth in $V$, the quotient module  $\OO_{U,x}$ admits a free resolution of $m-n+1$ terms. Therefore,
$$\EE xt^q_{\mathcal O_V}(\iota_*\mathcal G,\FF)_x\cong \mathrm{Ext}^q_{\OO_{V,x}}(\iota_*\GG_x,\FF_x)=0,$$
when $q\geq m-n+1$.

As a consequence, the term $E^{p,q}_2=0$ when $p+q>m-n+l$, so we get the Ext vanishing as that in the statement.
\end{proof}

Applying Lemma \ref{lem:caextvanish}, we obtain the following result which allows to rule out some destabilizing objects for $\Psi\rho^*E_{\Gamma}$.
\begin{lem}  
\label{lem:nosuppdim1sub}
Let $\Gamma$ be a non-degenerate smooth elliptic quintic spanning  a smooth cubic threefold $X$.
For any $F\in\Coh(\P^3,\BB_0)$ such that $\dim\mathrm{supp}(F)\leq 1$, we have 
$$\Hom_{\BB_0}(F,\Psi\rho^*E_{\Gamma})=0.$$
\end{lem}
\begin{proof}
By using the property of adjoint functors and Serre duality, we have
\begin{align}
    & \Hom_{\BB_0}(F,\Psi\rho^*E_{\Gamma})\cong (\Hom_{\BB_0}(\Psi\rho^*E_{\Gamma},F\otimes_{\BB_0}\BB_{-3}[3]))^* \\ 
    \cong  &(\Hom_{\mathcal O_{\tilde Y}}(\rho^*E_{\Gamma},\Phi(F\otimes_{\BB_0}\BB_{-3})[3])))^*\\
    \cong & \Hom_{\mathcal O_{\tilde Y}}(\Phi(F\otimes_{\BB_0}\BB_{-3}),\rho^*E_{\Gamma}\otimes K_{\tilde Y}[1]). \label{eq:homY}
\end{align}
Recall from \eqref{eq:defofphi} that  $\Phi$ is the right adjoint functor of $\Psi$. By Condition \ref{cond:choiceofL}(c) on the choice of $P$, the morphism $\pi:\tilde X\rightarrow \P^3$ is generically finite and only contracts finitely many lines. Thus we have 
$$\dim(\mathrm{supp}(\Phi(F\otimes_{\BB_0}\BB_{-3}))\cap \tilde X)\leq 1.$$
By Lemma \ref{lem:caextvanish}, we conclude that the  $\Hom_{\mathcal O_{\tilde Y}}$ in the formula (\ref{eq:homY}) is $0$.
\end{proof}
\begin{prop}
Let $\Gamma$ be a non-degenerate smooth elliptic quintic spanning  a smooth cubic threefold. Then the object $\Psi\rho^*(E_\Gamma)$ is in $\Coh^{-1}(\P^3,\BB_0)$ and   $\sigma_{\alpha,-1}$-stable for $\alpha \gg 0$.
\label{prop:stabofEgamma}
\end{prop}
\begin{proof}
\textbf{Step 1:}
By Condition \ref{cond:choiceofL} on the choice of $P$, the projection map from $\pi|_{\tilde X}:\tilde X\rightarrow \P^3$ is generically finite  except contracting (the transverse image of) finitely many lines that across $P$ on $X$. Note that $\rho^*E_\Gamma$ is locally free, hence by definition of $\Psi$ in \eqref{eq:defofPsi} and Lemma \ref{lem:Eisvb}, the object $\Psi(\rho^*E_\Gamma)$ is contained in the extension closure of $\{\mathcal{T}\mathrm{or}^{\leq 0},\Coh(\P^3,\BB_0)[1]\}$, where $\mathcal T\mathrm{or}^{\leq 0}$ consists of torsion $\BB_0$-modules supported on a $0$-dimensional locus. In particular, the object $\Psi\rho^*(E_\Gamma)$ sits in  the distinguished triangle
\begin{equation}\label{eq:torextforpregamma}
    G[1]\rightarrow \Psi(\rho^*E_\Gamma) \rightarrow T\xrightarrow{+},
\end{equation}
where $G\in \Coh(\P^3,\BB_0)$ and $T$ is a torsion $\BB_0$-modules supported on a $0$-dimensional locus. 

\textbf{Step 2:} To show that $\Psi\rho^*(E_\Gamma)\in\Coh^{-1}(\P^3,\BB_0)$, it is enough to show that for any torsion-free $\BB_0$-module $D$ with rank $4$ and slope $\mu_{\mathrm{slope}}(D)>-1$, the vanishing 
$$\Hom_{\BB_0}(D,G)=0$$
holds.
Suppose $\Hom_{\BB_0}(D,G)\neq 0$, then $\Hom_{\BB_0}(D[1],\Psi(\rho^*E_\Gamma))\neq0$. Taking the double dual of $D$ as that in Definition \ref{def:doubledual}, we have the distinguished triangle
$$D_{\mathrm{tor}}\rightarrow D[1]\rightarrow D^{**}[1]\xrightarrow{+},$$
where $D_{\mathrm{tor}}$ is a torsion $\BB_0$-module supported at a locus with dimension at most $1$. By Lemma \ref{lem:nosuppdim1sub}, we have $\Hom_{\BB_0}(D^{**}[1],\Psi(\rho^*E_\Gamma))\neq0$. By Lemma \ref{lem:rk4b0alg} and the fact that $\mu_{\mathrm{slope}}(\BB_0)=-\frac{5}4$,  we may assume $$D^{**}\cong \BB_i\text{ for some }i\geq 1.$$
By Lemma \ref{lem:homvanforEgamma}, this can never happen. As a summary, the object $\Psi(\rho^*E_\Gamma)$ is in $\Coh^{-1}(\P^3,\BB_0)$.

\textbf{Step 3:} To show that $\Psi(\rho^*E_\Gamma)$ is $\sigma_{\alpha,-1}$-stable for $\alpha \gg 0$, we need to rule out the possibility that 
\begin{enumerate}
    \item[(i)] $\Psi(\rho^*E_\Gamma)$ has a sub-torsion object which is a torsion $\BB_0$-module with support of dimension at most $1$;
    \item[(ii)] $\Psi(\rho^*E_\Gamma)$ has a quotient object $F$ in $\Coh^{-1}(\P^3,\BB_0)$ such that  $F$ is a $\sigma_{+\infty,-1}$-stable with rank $-4$ and $\mu_{\alpha,-1}(F)\leq\mu_{\alpha,-1}(\Psi(\rho^*E_\Gamma))$ for $\alpha\gg 0$.
\end{enumerate}
Case (i) cannot happen by Lemma \ref{lem:nosuppdim1sub}. As for Case (ii), let $K$ be the kernel of $\Psi(\rho^*E_\Gamma)\twoheadrightarrow F$ in $\Coh^{-1}(\P^3,\BB_0)$.  Consider the exact sequence in $\Coh(\P^3,\BB_0)$:
$$0\rightarrow \HH^{-1}(K)\rightarrow \HH^{-1}(\Psi(\rho^*E_\Gamma))\rightarrow \HH^{-1}(F)\rightarrow \HH^{0}(K)\rightarrow \HH^{0}(\Psi(\rho^*E_\Gamma))\rightarrow \HH^{0}(F)\rightarrow 0.$$
Note that the term $\HH^{0}(\Psi(\rho^*E_\Gamma))$ is supported on a $0$-dimensional locus or is zero. By Corollary \ref{cor:rk4stontheright}, $\HH^0(F)=0$ and the object  $$F=\HH^{-1}(F)[1]=\BB_{i}[1]$$ for some $i\leq 0$. By Definition \ref{def:weakstabfunction}, Proposition \ref{prop:constructstabcond} and Remark \ref{rem:charofpsi2lambda}, we have $$\mu_{\alpha,-1}(\BB_0[1])=2\alpha^2+\frac{1}8>\frac{16\alpha^2+7}{24}=\mu_{\alpha,-1}(\Psi(\rho^*E_\Gamma))$$ for $\alpha\gg 0$. Therefore $\BB_0[1]$ does not codestabilize $\Psi(\rho^*E_\Gamma)$. We may assume $F=\BB_i$ for some $i\leq -1$. 

Note that $K\in \Coh^{-1}(\P^3,\BB_0)$, so we have $\ch^{-1}_{\BB_0,1}(K)\geq 0$ which implies $\ch^{-1}_{\BB_0,1}(F)\leq 6$. Therefore, we may assume $F=\BB_i$ for some $i\geq -2$.

In either case of $i=-1,-2$, by Serre duality and the fact that $\Psi(\rho^*E_\Gamma)$ is an object in $ \Ku(\P^3,\BB_0)$, we have
$$\Hom_{\BB_0}(\Psi(\rho^*E_\Gamma),\BB_i[1])\cong\left(\Hom_{\BB_0}(\BB_{i+3},\Psi(\rho^*E_\Gamma)[2])\right)^*=0.$$
 Therefore, Case (ii) can neither happen. We conclude that $\Psi(\rho^*E_\Gamma)$ is $\sigma_{\alpha,-1}$-stable for $\alpha \gg 0$.
\end{proof}
\begin{comment}
As for the only remaining case,
\begin{align*}
    & \Hom_{\BB_0}(\BB_0[1],\Psi(\sigma^*E_\Gamma))\\ \cong &\Hom_{\BB_0}(\Psi(\OO_{\tilde Y}(H-h)),\Psi(\sigma^*E_\Gamma))\\\cong &\Hom_{\OO_{\tilde Y}}(\OO_{\tilde Y}(H-h),\Phi(\Psi(\sigma^*E_\Gamma)))\\
    \cong &\Hom_{\OO_{\tilde Y}}(\OO_{\tilde Y}(H-h),\mathsf{L}_{\OO_{\tilde Y}(-h)}\sigma^*E_\Gamma)
    \subset \Hom_{\OO_{X}}(\OO_X,E_\Gamma)=0.
\end{align*}
The last equality is by noticing that the curve $\Gamma$ spans $\P^4$. Therefore,  $$\Hom_{\OO_X}(\OO_X,E_\Gamma)=\Hom_{\OO_X}(\OO_X,\II_{\Gamma/X}(H))=0.$$
\end{comment}

\subsection{Tilt-Stability of $\Psi\rho^*(E_C)$} \label{subsec_tiltstabC}
Let $C$ be a smooth conic curve on $Y$. Similarly to the case of non-degenerate elliptic quintics, we now study the $\sigma_{\alpha, -1}$-stability of the object $\Psi\rho^*(E_C)$ for $\alpha$ large enough. We choose the blown-up line $L$ for $\rho:\tilde Y\rightarrow Y$ such that:
\begin{cond}\label{cond:LforEC}
\begin{enumerate}
    \item [\rm(a)] the line $L$ does not intersect the projective plane spanned by $C$;
    \item [\rm(b)] the plane spanned by $L$ and a generic point  on $C$ intersects $Y$ at the union of $L$ and a smooth conic curve.
\end{enumerate}
\end{cond}
 By Lemma \ref{lem:formofEC}, the object $\Psi\rho^*(E_C)$ sits in the distinguished triangle:
\begin{equation}\label{eqn:formularofpsirhoEC}
    \Psi\rho^*\left(\OO_Y(-H)^{\oplus 2}[1]\right) \to \Psi\rho^*(E_C) \to \Psi\rho^*(\mathsf L_{\OO_Y}(\theta_C(H))[-1]).
\end{equation}

By \eqref{eq:imageofOHs},  the triangle can be simplified as
\begin{equation}
\label{eq_defE_CasB_0mod}
\BB_{-1}^{\oplus 2}[1] \to \Psi\rho^*(E_C) \to T_C,    
\end{equation}
where $$T_C:=\Psi\rho^*(\mathsf L_{\OO_Y}(\theta_C(H))[-1])=\Psi\rho^*(\theta_C(H)[-1]).$$ The second equality is by noticing that $\Psi\rho^*(\OO_Y)=0$.  By the choice of $L$ as that in Condition \ref{cond:LforEC}(a), the image $C'\coloneqq\pi(\rho^{-1}(C))$ is a smooth conic in $\P^3$. By the definition of $\Psi$ in \eqref{eq:defofPsi} and Lemma \ref{lem:Eisvb}, the object $T_C$ is a torsion $\BB_0$-module supported on $C'$.
\begin{lem}\label{lem:TCasasheaf}
Adopt the notation as above.
\begin{enumerate}
    \item As a $\OO_{\P^3}$-coherent sheaf, $\Forg(T_C)\cong \OO_{C'}^{\oplus 2}$. 
\item A torsion $\BB_0$-module with $C'$ as its support has rank at least $2$.  
\end{enumerate}
In particular, the sheaf $T_C$ is indecomposable as a $\BB_0$-module.
\end{lem}
\begin{proof}
(1). Note that $\Psi\rho^*(E_C)$ is an object in $\Ku(\P^3,\BB_0)$, so we have $$0=\Hom_{\BB_0}(\BB_1,\Psi\rho^*(E_C)[i])=\Hom_{\BB_0}(\BB_0,T_C\otimes_{\BB_0}\BB_{-1}[i])$$ for every $i\in \Z$. Denote the embedding map by $\iota:C'\rightarrow \P^3$. By the definition of $\Psi$ in \eqref{eq:defofPsi} and Lemma \ref{lem:Eisvb}, the sheaf $\Forg(T_C)=\iota_*(\FF_{C'})$ for some rank $2$ locally free sheaf $\FF_{C'}$ on $C'$.  By \eqref{eq:forgetful}, we have$$0=\Hom_{\OO_{\P^3}}(\OO_{\P^3},\Forg(T_C\otimes_{\BB_0}\BB_{-1})[i])=\Hom_{\OO_{C'}}(\OO_{C'},\FF_{C'}\otimes \theta_{C'}[i])$$ for every $i\in\Z$. This can only happen when $\FF_{C'}\cong\OO_{C'}^{\oplus 2}$.

(2). By the choice of $L$ as that in Condition \ref{cond:LforEC}(b), the $\BB_0$-algebra structure as that in \eqref{eq:clifalg} on a generic point on $C$ is isomorphic to $\mathrm{Mat}_{2\times 2}(\C)$, the $2$ by $2$ complex matrices, as a $\C$-algebra. Since a $\mathrm{Mat}_{2\times 2}(\C)$-module is at least with dimension $2$ as a $\C$-vector space, there is no torsion  $\BB_0$-module supported on $C'$ with rank $1$. In particular, $T_C$ is indecomposable as a $\BB_0$-module.
\end{proof}

\begin{lem}
\label{lem:basicforhomivan}
If $F\in\Coh(\P^3,\BB_0)$ is such that $\dim\mathrm{supp}(F)\leq m$, then $\Hom_{\BB_0}(F,\BB_i[j])=0,$ for $j\leq 2-m$ and all $i\in\Z$.
\end{lem}
\begin{proof}
By Serre duality and \eqref{eq:forgetful},
\begin{align*}
\Hom_{\BB_0}(F,\BB_i[j])& \cong(\Hom_{\BB_0}(\BB_0,F\otimes_{\BB_0}\BB_{-3-i}[3-j]))^* \\
&\cong(\Hom_{\OO_{\P^3}}(\OO_{\P^3},\Forg(F\otimes_{\BB_0}\BB_{-3-i})[3-j]))^*=0
\end{align*}
when $3-j\geq 1+m$.
\end{proof}
\begin{lem}
\label{lem:nosuppdim1subEC}
For any $F\in\Coh(\P^3,\BB_0)$ such that $\dim\mathrm{supp}(F)\leq 1$, we have 
$$\Hom_{\BB_0}(F,\Psi(\rho^*E_C))=0.$$
\end{lem}
\begin{proof}
%As  $F$ is supported in dimension $\leq 1$, by Lemma \ref{lem:basicforhomivan}, we have
%$\Hom_{\BB_0}(F,\BB_{-1}[1])=0$.
We first show that for any sub-$\BB_0$-module $F$ of $T_C$, the statement holds. By Lemma \ref{lem:TCasasheaf}, the sheaf $F$ is supported on $C'$, locally free with rank $2$ as a sheaf on $C'$. Moreover, any non-zero morphism $f:F\rightarrow T_C$ is injective. Applying $\Hom_{\BB_0}(-,\BB_{-1}^{\oplus 2})$ to the short exact sequence \begin{equation}
\label{eq_FandF'}
0 \to F \xrightarrow{f} T_C \to F' \to 0,
\end{equation} by Lemma \ref{lem:basicforhomivan}, we have $$\Hom_{\BB_0}(F',\BB_{-1}^{\oplus 2}[2])=0\text{ and }\Hom_{\BB_0}(F,\BB_{-1}^{\oplus 2}[1])=0.$$ Thus the morphism $$f\circ -\colon\Hom_{\BB_0}(T_C,\BB_{-1}^{\oplus 2}[2]) \to \Hom_{\BB_0}(F, \BB_{-1}^{\oplus 2}[2])$$ is injective. In other words, the composition of $\mathsf{ev}:T_C \to \BB_{-1}^{\oplus 2}[2]$ with any non-zero $f$ is a non-zero morphism in $\Hom_{\BB_0}(F, \BB_{-1}^{\oplus 2}[2])$. Therefore, in \eqref{eqn:formularofpsirhoEC}, any non-zero morphism $f:F\rightarrow T_C $ cannot lift to a morphism from $F$ to $\Psi\rho^*(E_C)$. Note that $\Hom_{\BB_0}(F,\BB_{-1}[1])=0$ by Lemma \ref{lem:basicforhomivan}, so the statement holds for any sub-$\BB_0$-module $F$ of $T_C$.

As for an arbitrary $F$ with $\dim\supp{F}\leq 1$, we make induction on its $\ch_2$. Let $g$ be a morphism in $\Hom_{\BB_0}(F,\Psi(\rho^*E_C))$. Applying $\Hom_{\BB_0}(F,-)$ to \eqref{eqn:formularofpsirhoEC}, we have
$$..\to \Hom_{\BB_0}(F,\BB_{-1}[1]^{\oplus 2})\to \Hom_{\BB_0}(F,\Psi(\rho^*E_C))\to \Hom_{\BB_0}(F,T_C)\to..$$
and since $\Hom_{\BB_0}(F,\BB_{-1}[1])=0$, the morphism $g$ is mapped to a  morphism $g'$. 

Suppose $g\neq 0$, then $g'\neq 0$ and $\Hom_{\BB_0}(\mathrm{im}(g'),\Psi(\rho^*E_C))=0$ as $\mathrm{im}(g')$ is a submodule of $T_C$. By induction, $\Hom_{\BB_0}(\mathrm{ker}(g'),\Psi(\rho^*E_C))=0$. Therefore, $\Hom_{\BB_0}(F,\Psi(\rho^*E_C))=0$, which contradicts $g\neq 0$. 
\end{proof}

\begin{comment}

Assume now that $F$ has rank $1$ on $C'$. Then $T_C$ sits in a sequence of the form \eqref{eq_FandF'} where $F'$ has rank $1$ as a sheaf on $C'$. Recall that by definition $T_C=\pi_*(\EE(h)|_{\sigma^{-1}(C)} \otimes L_C)$. Computing the degree of $\EE|_C$, we deduce that $\Forg(T_C)\cong \OO_{\P^1}^{\oplus 2}$. Thus $\Forg(F) \cong \OO_{\P^1}(-k)$ for $k \geq 0$. If $\Forg(F)$ injects in a copy of $\OO_{\P^1}$, then $F' \cong F'_{\text{tor}} \oplus L$, where $\Forg(F'_{\text{tor}})$ is a torsion sheaf on $\P^1$ and $\Forg(L) \cong \OO_{\P^1}$. Thus $F'_{\text{tor}}$ does not map to $\BB_{-1}^{\oplus 2}[2]$ and the morphism from $L$ to $\BB_{-1}^{\oplus 2}[2]$ lifts to the same morphism from the corresponding copy of $\OO_{\P^1}$ in $T_C$. As a consequence, we have that the morphism $T_C \to \BB_{-1}^{\oplus 2}[2]$ lifts to a non-zero morphism $F \to \BB_{-1}^{\oplus 2}[2]$, as it is induced by a morphism from the other copy of $\OO_{\P^1}$ in $T_C$. Otherwise, $F'$ is a line bundle over $\P^1$ and then we have $\Hom_{\BB_0}(F', \BB_{-1}^{\oplus 2}[2])=0$. In both situations, we deduce the proof of the claim. 

As a consequence, we have $\Hom_{\BB_0}(F,E_C)=0$, as we wanted.

\end{comment}

\begin{prop}
Let $C$ be a smooth conic curve. Then the object $\Psi\rho^*(E_C)$ is in $\Coh^{-1}(\P^3,\BB_0)$ and   $\sigma_{\alpha,-1}$-stable for $\alpha \gg 0$.
\label{prop:stabofEC}
\end{prop}
%The argument follows as that for the stability of $\Psi\sigma^*(E_\Gamma)$.
\begin{proof}
By (\ref{eq_defE_CasB_0mod}), the object $\Psi\rho^*(E_C)$ is in $\Coh^{-1}(\P^3,\BB_0)$.   To show that $\Psi(\rho^*E_C)$ is $\sigma_{\alpha,-1}$-stable for $\alpha \gg 0$, we need to rule out the possibility that 
\begin{enumerate}
    \item[(i)] $\Psi(\rho^*E_C)$ has a sub-torsion object which is a torsion $\BB_0$-module with support of dimension at most $1$;
    \item[(ii)] $\Psi(\rho^*E_C)$ has a quotient object $F$ in $\Coh^{-1}(\P^3,\BB_0)$ such that $F$ is a $\sigma_{+\infty,-1}$-stable with rank $-4$ and $\mu_{\alpha,-1}(F)\leq\mu_{\alpha,-1}(\Psi(\rho^*E_\Gamma))$ for $\alpha\gg 0$.
\end{enumerate}
Case (i) cannot happen by Lemma \ref{lem:nosuppdim1subEC}. As for Case (ii), let $K$ be the kernel of $\Psi(\rho^*E_C)\twoheadrightarrow F$ in $\Coh^{-1}(\P^3,\BB_0)$, then we have the exact sequence in $\Coh(\P^3,\BB_0)$:
$$0\rightarrow \HH^{-1}(K)\rightarrow \BB_{-1}^{\oplus 2}\rightarrow \HH^{-1}(F)\rightarrow \HH^{0}(K)\rightarrow T_C \rightarrow \HH^{0}(F)\rightarrow 0.$$
By Corollary \ref{cor:rk4stontheright}, we may assume  $\HH^{-1}(F)$ is  $\BB_{i}[1]$ for some $i\leq 0$. Note that $\BB_0[1]$ has a larger slope than $\Psi(\rho^*E_\Gamma)$ with respect to $\sigma_{\alpha,-1}$, so we have $i\leq -1$. Since Hom$_{\BB_0}(\BB_{-1}^{\oplus 2},\BB_i)=0$ for every $i\leq -2$, the sheaf $\HH^{-1}(F)$ can only be $\BB_{-1}$ as well as $\HH^{-1}(K)$. Hence we have the sequence  in $\Coh(\P^3,\BB_0)$:
$$0\rightarrow  \HH^{0}(K)\rightarrow T_C \rightarrow \HH^{0}(F)\rightarrow 0.$$
 By Lemma \ref{lem:TCasasheaf}, the sheaf $\HH^0(F)$ is  either with $0$-dimensional support, or supported on $C'$, locally free of rank  $2$ as a sheaf on $C'$. The second case cannot happen since the slope of $F$ would be larger than that of $\Psi(\rho^*E_C)$. By Corollary \ref{cor:rk4stontheright}, $\HH^0(F)=0$. In other words, $F=\BB_{-1}[1]$.
 %\todo{cannot get rid of the semi-stability, what assumption should we make?}
 
 Since $\Psi(\rho^*E_C)\in \Ku(\P^3,\BB_0)$, by Serre duality, we have
$$\Hom_{\BB_0}(\Psi(\rho^*E_C),F)\cong(\Hom_{\BB_0}(\BB_{2},\Psi(\rho^*E_C)[2]))^*=0.$$
 Therefore, Case (ii) can neither happen. The object $\Psi(\rho^*E_C)$ is $\sigma_{\alpha,-1}$-stable for $\alpha \gg 0$.
\end{proof}

\subsection{No actual walls for  $\Psi\rho^*(E_\Gamma)$ and $\Psi\rho^*(E_C)$}\label{sec:nowall}

By Propositions \ref{prop:stabofEgamma} and \ref{prop:stabofEC}, we have the $\sigma_{\alpha,-1}$-stability of $\Psi\rho^*(E_\Gamma)$ ($\Psi\rho^*(E_C)$) for $\alpha\gg 0$. In this section, we show that $\sigma_{\alpha,-1}$-stable objects in $\Ku(\P^3,\BB_0)$ with this character cannot be destabilized when $\alpha$ decreases. 

We first list the character $\ch^{-1}_{\BB_0,\leq 2}$ of all possible destabilizing objects with respect to the weak stability conditions $\sigma_{\alpha,-1}$. Recall that the rank of $\BB_0$-modules on $\P^3$ is always a multiple of $4$. We can write the characters of potential destabilizing subobjects and quotient objects for $\Psi\rho^*(E_\Gamma)$ and $\Psi\rho^*(E_C)$ as
\begin{equation}
\label{chF}
\ch^{-1}_{\BB_0,\leq 2}(\Psi\rho^*(2\lambda_1+2\lambda_2))=(-8,6,\frac{14}8)=(4a,b,\frac{c}{8})+(-8-4a,6-b,\frac{14}8-\frac{c}{8}),
\end{equation}
where $a, b, c \in \Z$. These characters have to satisfy the following conditions:
\begin{enumerate}
    \item[(a)] The two characters have non-negative discriminant $\dlt$ by Proposition \ref{prop:constructstabcond}.
    \item[(b)] The two characters should be integral combinations of the characters of $\ch^{-1}_{\BB_0,\leq 2}(\BB_i)$ for $i=-1,0,1$ by restriction to $\Coh(\P^2,\BB_0|_{\P^2})$ and \cite[Proposition 2.12]{BMMS:Cubics}. In particular, the set 
    \begin{equation}\label{eq:basisforch2}
    \{(4,1,\frac{1}8),(0,2,0),(0,0,1)\}
         \end{equation} forms a $\Z$-linear basis for all possible characters. 
    \item[(c)] There exists $\alpha>0$ such that the two characters have the same slope with respect to $\sigma_{\alpha,-1}$. In particular, both $b$ and $6-b>0$. 
    \item[(d)] Without loss of generality, we may assume that the character $(4a,b,\frac{c}8)$ is the character of a destabilizing subobject. The equivalent numerical assumption is  $$\frac{4a}{b}>\frac{\rk(\Psi\rho^*(2\lambda_1+2\lambda_2))}{\ch^{-1}_{\BB_0,1}(\Psi\rho^*(2\lambda_1+2\lambda_2))}=-\frac{4}3.$$ 
    %\item The ordinary Chern character of objects in $\Db(\P^3)$ truncated to degree $2$ is represented by a triple $(R,C,D/2)$, where $C$ and $D$ are integers of the same parity. Thus, the two characters have the form
%$$(R,C,\frac{D}{2})(1,0,-\frac{11}{32})(1,1,\frac{1}{2})=(R,C+R,\frac{D}{2}+C+\frac{5}{32}R).$$ 
\end{enumerate}

Using these conditions, by a standard computation we obtain the following result.

\begin{prop}
\label{walls}
All possible solutions of \eqref{chF} are:
\begin{enumerate}
\item for $\alpha=\frac{\sqrt{17}}{4}$, $a=0$, $b=2$, $c=16$; 
\item for $\alpha=\frac{\sqrt 5}{4}$, 
\begin{enumerate}
\item[(i)] $a=-1$, $b=5$, $c=15$;
\item[(ii)] $a=0$, $b=4$, $c=16$;
\item[(iii)] $a=0$, $b=2$, $c=8$;
\end{enumerate}
\item for $\alpha=\frac{1}{4}$, $a=1$, $b=3$, $c=9$.
\end{enumerate}
\end{prop}
\begin{proof}
We sketch the steps of the computation here. The first step is to rule out the `higher rank' wall case. Namely, by the non-negativity condition (a), (c) and (d), one may deduce that $-12\leq -8-4a\leq 0$. Then by condition (b) and (d), the possible pairs $(a,b)$ are $(-1,5)$, $(0,2)$, $(0,4)$, $(1,1)$, $(1,3)$, and $(1,5)$. By condition (a) again, one can list all possible triples of $(a,b,c)$.
\end{proof}
 %We would like to point out that, by \eqref{eq:basisforch2}, in each case other than (2.b), both characters $(4a,b,\frac{c}{8})$ and $(-8-4a,6-b,\frac{14}8-\frac{c}{8})$ are `primitive' on the wall. In other words, a $\sigma_{\alpha,-1}$-semistable object with any of these characters is $\sigma_{\alpha,-1}$-stable.

\begin{prop}\label{prop:nowallforobjinKU}
Let $E$ be a $\sigma_{\alpha_0,-1}$-stable object in $\Ku(\P^3,\BB_0)$ and $\Coh^{-1}(\P^3,\BB_0)$ with $\chbl(E)=(-8,6,\frac{14}8)$. Then $E$ is $\sigma_{\alpha,-1}$-stable for any $\alpha\leq \alpha_0$.
\end{prop}
\begin{proof}
Suppose $E$ becomes strictly semistable with respect to $\sigma_{\alpha,-1}$ for some $0<\alpha<\alpha_0$. 
By Proposition \ref{walls}, this may happen when $\alpha=\frac{1}{4}, \frac{\sqrt 5}{4}$ or $\frac{\sqrt{17}}{4}$. Let us denote the destabilizing sequence in $\Coh^{-1}(\P^3,\BB_0)$ as follows:
\begin{align}
0\rightarrow S\rightarrow E\rightarrow Q\rightarrow 0,\label{eq:desE}
\end{align}
where $S$ and $Q$ are $\sigma_{\alpha,-1}$-semistable objects with characters as those in Proposition \ref{walls}.\\%\todo{I have moved this paragraph before Step I, since the notation introduced here is common to all the other steps. Of course, feel free to come back to the previous version if you prefer :-)} 

\textbf{Step I:} We get rid of two cases when the destabilizing object is $\BB_0^{\oplus a}[1]$. 

When $\alpha=\frac{\sqrt 5}{4}$, if Case ({2.i}) or ({2.ii}) in Proposition \ref{walls} happens, then the Chern character of the quotient object $Q$ is 
$$\ch_{\BB_0,\leq 2}^{-1}(Q)=(-4,1,-\frac{1}8) \text{ or }(-8,2,-\frac{1}4).$$
By Lemma \ref{lem:dlt0obj}, the quotient object $Q$ is either $\BB_0[1]$ or $\BB_0^{\oplus 2}[1]$. In either case, we would have
$$\Hom(\BB_3,E[2])\cong(\Hom(E,\BB_0[1]))^*\neq 0,$$
which contradicts the assumption that $E\in \Ku(\P^3,\BB_0)$.\\

\textbf{Step II:} We show that $\Hom(\BB_j,Q[i])=0$ for $i\geq 1$ and $j=1,2,3$.

Now there are three cases in Proposition \ref{walls} left. In Case (1) and (3), it is a direct computation that any $\sigma_{\frac{\sqrt{17}}4,-1}$-semistable (resp. $\sigma_{\frac{1}4,-1}$) objects with character $(0,2,2)$ and $(-8,4,-\frac{1}4)$ (resp. $(4,3,\frac{9}8)$ and $(-12,3,\frac{5}8)$) are $\sigma_{\frac{\sqrt{17}}4,-1}$-stable (resp. $\sigma_{\frac{1}4,-1}$). Both $S$ and $Q$ are $\sigma_{\alpha,-1}$-stable in these two cases. In Case (2.iii), the object $S$ with character $(0,2,8)$ is also $\sigma_{\frac{\sqrt{5}}4,-1}$-stable. If $Q$ is strictly $\sigma_{\frac{\sqrt5}4,-1}$-semistable, we may reduce to either Case (2.i) or (2.ii). Therefore, in any of the remaining cases, we may assume both $S$ and $Q$ are  $\sigma_{\alpha,-1}$-stable.

For each $\BB_j$, $1\leq j\leq 3$, apply $\Hom(\BB_j,-)$ to the sequence (\ref{eq:desE}). Since $E\in \Ku(\P^3,\BB_0)$, we have 
\begin{align}
\label{eq:homsq}\Hom(\BB_j,S[i+1])\cong \Hom(\BB_j,Q[i])
\end{align}
for all $i\in \Z$. 

We first show that $\Hom(\BB_j,S[i+1])=0$ for $i\geq 1$. In   Case (1) of Proposition \ref{walls},  the object $S$ is $\sigma_{\frac{\sqrt{17}}{4},-1}$-stable, we may let $\alpha_1=\frac{\sqrt{17}}{4}$. In Case (2.iii) of Proposition \ref{walls}, as $S$ is $\sigma_{\frac{\sqrt 5}{4},-1}$-stable, the object $S$ is $\sigma_{\alpha_1,-1}$-stable for some $\alpha_1<\frac{\sqrt 5}{4}$. In Case (3) of Proposition \ref{walls}, we may let $\alpha_1=\frac{1}4$. By the choice  of $\alpha_1$ in each case, we always have
\begin{align}\label{eqn:compare}
 \mu_{\alpha_1,-1}(S)=\begin{cases}
 1 & \text{Case (1)}\\
 \frac{1}2 & \text{Case (2.iii)}\\
 \frac{-4\alpha_1^2+9}{24} & \text{Case (3)}
 \end{cases}>\frac{4\alpha_1^2-1}8=\mu_{\alpha_1,-1}(\BB_0[1])\geq \mu_{\alpha_1,-1}(\BB_{j}[1])
\end{align}
for $j\leq 0$. Note that both $S$ and $\BB_j[1]$ are in $\Coh^{-1}(\P^3,\BB_0)$ and $\sigma_{\alpha_1,-1}$-stable for $j\leq 0$. By \eqref{eqn:compare} and Serre duality,
\begin{align}
\Hom(\BB_{j+3},S[i+1])\cong (\Hom(S,\BB_{j}[2-i]))^*=0
\end{align}
for any $j\leq 0$, $i\geq 1$. By \eqref{eq:homsq}, we have $\Hom(\BB_j,Q[i])=0$ for $i\geq 1$ and $1 \leq j\leq 3$. \\

\textbf{Step III:} We show that $\Hom(\BB_j,Q[i])=0$ for $i\leq -1$ and $j=1,2,3$, or $i=0$ and $j=2,3$.

As $\BB_1$, $\BB_2$, $\BB_3$ and $Q$ are in the heart $\Coh^{-1}(\P^3,\BB_0)$, we have $\Hom(\BB_j,Q[i])=0$ for any $j=1,2,3$ and $i\leq -1$. %\emph{By (\ref{eq:homsq})} \todo{I would modify in this way.}
Together with Step II, this implies $\Hom(\BB_j,Q[i])$ may be nonzero only when $i=0$. In   Case (1) of Proposition \ref{walls}, as $\ch_{\BB_0,\leq 2}^{-1}(Q)=(-8,4,\frac{1}4)$, the object $Q$ is $\sigma_{\alpha_1,-1}$-stable for $\alpha_1\in(\frac{1}4,\frac{\sqrt{17}}{4}]$, we may let $\alpha_1=1$. In Case (2.iii) of Proposition \ref{walls}, as $Q$ is $\sigma_{\frac{\sqrt 5}{4},-1}$-stable, we may let $\alpha_1=\frac{\sqrt 5}{4}$. In Case (3) of Proposition \ref{walls}, the object $Q$ is $\sigma_{\alpha_1,-1}$-stable for some $\alpha_1<\frac{1}4$. By the choice  of $\alpha_1$ in each case, we always have
\begin{align}
 \mu_{\alpha_1,-1}(Q)=\begin{cases}
 \frac{4\alpha_1^2-1}{16} & \text{Case (1)}\\
 \frac{4\alpha_1^2+3}{16} & \text{Case (2.iii)}\\
 \frac{12\alpha_1^2+5}{24} & \text{Case (3)}
 \end{cases}<  \frac{-4\alpha_1^2+9}{24}= \mu_{\alpha_1,-1}(\BB_{2}) \leq \mu_{\alpha_1,-1}(\BB_{j})
\end{align}
for $j\geq 2$. Therefore, 
\begin{align}
\Hom(\BB_j,Q)=0 \text{ for } j\geq 2.
\end{align}

\textbf{Step IV:} We show that the character of $Q$ or  $\mathsf L_{\BB_1}Q$ cannot be in $\Ku(\P^3,\BB_0)$.

Now $\Hom(\BB_1,Q)$ is the only possible non-zero space among all $\Hom(\BB_j,Q[i])$ for $j=1,2,3$, $i\in \Z$. Therefore, the object $$\mathsf L_{\BB_1}Q=\mathrm{Cone}(\BB_1\otimes \Hom(\BB_1,Q)\rightarrow Q)$$ is in $\Ku(\P^3,\BB_0)$. By \eqref{eqn:charofKuobj} in Definition and Proposition \ref{prop:stabonku}, the character   $$\ch_{\BB_0,\leq 2}^{-1}(\mathsf L_{\BB_1}Q)\in\ch_{\BB_0,\leq 2}^{-1}(\Ku(\P^3,\BB_0))\subset \{(a,b,-\frac{7}{32}a)|a,b\in \R\}.$$

On the other hand, in any case of Proposition \ref{walls}, we have $\rk(Q)< 0$ and $\frac{\ch_{\BB_0, 2}^{-1}(Q)}{\rk(Q)}\geq -\frac{3}{32}$.  Note that $\ch_{\BB_0,\leq 2}^{-1}(\BB_1[1])=(-4,-1,-\frac{1}8)$, so we have 
$$\frac{\ch_{\BB_0, 2}^{-1}(\mathsf L_{\BB_1}Q)}{\rk(\mathsf L_{\BB_1}Q)}=\frac{\ch_{\BB_0, 2}^{-1}(Q)+\hom(\BB_1,Q)\ch_{\BB_0, 2}^{-1}(\BB_1[1])}{\rk(Q)+\hom(\BB_1,Q)\rk(\BB_1[1])}\geq -\frac{3}{32}>-\frac{7}{32}.$$
We get the contradiction. Therefore, the object $E$ does not become strictly $\sigma_{\alpha,-1}$-semistable for any $\alpha\leq \alpha_0$.
\end{proof}

\begin{thm}\label{thm:EgammaECinkustab}
Let $\Gamma$ be a non-degenerate elliptic quintic curve  spanning a smooth cubic threefold on $Y$. Let $C$ be a smooth conic curve on $Y$. Let $\sigma$ be the stability condition on $\Ku(Y)$ as that in Proposition and Definition \ref{prop:stabonku} and Remark \ref{rem:noL}. Then the objects $E_\Gamma$ and $E_C$ are $\sigma$-stable.
\end{thm}
\begin{proof}
By Proposition \ref{prop:stabofEgamma}, \ref{prop:stabofEC}
and \ref{prop:nowallforobjinKU}, the objects  $\Psi\rho^*(E_\Gamma)$ and $\Psi\rho^*(E_C)$  are in the heart $\Coh^{-1}(\P^3,\BB_0)$ and   $\sigma_{\alpha,-1}$-stable for every  $\alpha > 0$. Note that $$\mu_{Z_{\alpha,-1}}(E_\Gamma)=\mu_{Z_{\alpha,-1}}(E_C)=\frac{16\alpha^2+7}{24}>0.$$ Thus $\Psi\rho^*(E_\Gamma)$ and $\Psi\rho^*(E_C)$ are in $\left(\Coh^{-1}(\P^3, \BB_0)\right)^0_{\sigma_{\alpha,-1}}\bigcap \Ku(\P^3,\BB_0)$.

By Lemma \ref{lem:nosuppdim1sub} and Lemma \ref{lem:nosuppdim1subEC}, both $\Psi\rho^*(E_\Gamma)$ and $\Psi\rho^*(E_C)$ satisfy the conditions as those in Lemma \ref{lem:tiltstabtokustab}. Therefore, they are stable with respect to the stability condition as that defined in Proposition \ref{prop:stabonku}. By Proposition \ref{prop:equivalentofKu} and Remark \ref{rem:noL}, both of them are $\sigma$-stable.
\end{proof}

\subsection{Example of $\CC_{12}$}\label{sec:C12} We give an example when $E_\Gamma$ is not expected to be $\sigma$-stable.

Denote by $\CC_{12}$ the divisor in the moduli space of cubic fourfolds parametrizing cubic fourfolds containing a rational cubic scroll \cite[Section 4.1.2]{Hassett-specialcubic4fold}.
Let $\Gamma$ be a non-degenerate elliptic quintic curve in $\P^5$; then $\Gamma$ is contained in a rational cubic scroll $\Sigma \subset \langle \Gamma \rangle$ (see \cite[Lemma 6.11]{HarrisRothStarr}). Assume that $\Sigma \subset Y$ for some smooth cubic fourfold in $\P^5$, in particular, the fourfold $Y$ is in $\CC_{12}$. Consider the cubic threefold $X:= \langle \Gamma \rangle \cap Y$, which contains $\Sigma$ by our assumption. We point out that such $X$ cannot be smooth.

\begin{lem}\label{lem:C12destabEGamma}
Let  $\Gamma$ be a non-degenerate elliptic quintic curve contained in a cubic scroll $\Sigma$ in $Y$. Then the object $\II_{\Sigma/X}(H)$ is in $\Ku(Y)$. If %$\Psi\rho^*(\II_{\Sigma/X}(H))$ is in the heart $\Coh^{-1}(\P^3,\BB_0)$ (in particular if 
$\II_{\Sigma/X}(H)$ is $\sigma$-stable, then 
$E_\Gamma$ is not $\sigma$-stable.
\end{lem}
\begin{proof}
%Note that we have the exact sequence
%$$0 \to \II_{\Sigma/X} \to \II_{\Gamma/X} \to \II_{\Gamma/\Sigma} \to 0.$$
%In particular, $\II_{\Sigma/X}(H)$ maps to $\II_{\Gamma/X}(H)$.
Consider the exact sequence $$0\rightarrow \II_{\Sigma/X}(H)\to \OO_X(H)\to \OO_\Sigma(H)\to 0.$$ 
Applying $\Hom(\OO_Y(mH),-)$ to the sequence for $m=0,1,2$, it is easy to observe that $\II_{\Sigma/X}(H)$ is in $\Ku(Y)$. In particular, by Serre duality, we have 
\begin{equation}\label{eqn:Isigmaku}
    \II_{\Sigma/X}(H)\in\!^\perp\!\langle \OO_Y(-2H),\OO_Y(-H)\rangle. 
\end{equation}
Recall from Definition \ref{def:EGammaEC} that:
$$E_\Gamma=\mathsf{pr}(\II_\Gamma(H))=\mathsf{R}_{\OO_Y(-H)}\mathsf{R}_{\OO_Y(-2H)}\mathsf L_{\OO_Y} (\II_\Gamma(H))=\mathsf{R}_{\OO_Y(-H)}\mathsf{R}_{\OO_Y(-2H)}( \II_{\Gamma/X}(H)).$$
 By \eqref{eqn:Isigmaku},
\begin{align*}
    \Hom(\II_{\Sigma/X}(H),E_\Gamma)& =\Hom(\II_{\Sigma/X}(H),\mathsf{R}_{\OO_Y(-H)}\mathsf{R}_{\OO_Y(-2H)}( \II_{\Gamma/X}(H)))\\ & \cong\Hom(\II_{\Sigma/X}(H), \II_{\Gamma/X}(H))\neq 0.
\end{align*}

Note that
$$\ch(\II_{\Sigma/X}(H))= \ch(\lambda_1) + \ch(\lambda_2) + s,$$%=H-\frac{H^2}{2}-\frac{1}{6}H^3+\frac{3}{8}+s,$$
where $s=H^2- \Sigma$ is a class in $H^{2,2}(Y,\Z)_{\text{prim}}$. In particular,  $H^2s=0$ and $\mu_\sigma(\II_{\Sigma/X}(H))=\mu_\sigma(E_\Gamma)$. Therefore, if $\Psi\rho^*(\II_{\Sigma/X}(H))$ is in $\Coh^{-1}(\P^3,\BB_0)$, then it will destabilize  $E_\Gamma$ with respect to $\sigma$.
\begin{comment}
We finish our claim by showing that $\Psi\rho^*(\II_{\Sigma/X}(H))$ is in $\Coh^{-1}(\P^3,\BB_0)$. Note that $\II_{\Sigma/X}(H)$ sits in the distinguished triangle:
$$\OO_{\Sigma}(H)[-1]\to \II_{\Sigma/X}(H) \to \OO_X(H).$$
The sheaf $\OO_X(H)$ sits in the distinguished triangle:
$$\OO_X \to \OO_X(H)\to \OO_S(H),$$
where $S$ is a smooth cubic surface. We may choose the blow-up $L$ on $Y$ not intersecting with $\Sigma$ nor $S$. By definition of $\Psi$ in \eqref{eq:defofPsi} and Lemma \ref{lem:Eisvb}, the objects $\Psi\rho^*()$
 if $\Psi\rho^*(E_\Gamma)$ is in $\Coh^{-1}(\P^3,\BB_0)$ (when it is not, $E_\Gamma$ is not $\sigma$-stable automatically). The object $\II_{\Sigma/X}(H)$ sits in the distinguished triangle:
$$\mathsf{pr}(\II_{\Gamma/\Sigma}(H)[-1])\to \II_{\Sigma/X}(H) \to \mathsf{pr}(\II_{\Gamma/X}(H))= E_\Gamma.$$
Note that $\II_{\Gamma/\Sigma}(H)=\OO_\Sigma(-D-F)$ where $D$ is the directrix and $F$ is the line of ruling. It is then a direct computation that $\mathsf{pr}(\II_{\Gamma/\Sigma}(H)[-1])$ sits in the distinguished triangle:
$$\OO_X(-H)[1]\to \mathsf{pr}(\II_{\Gamma/\Sigma}(H)[-1])\to \II_{\Gamma/\Sigma}(H)[-1]\xrightarrow{+}.$$
Note that $\Psi\rho^*(\II_{\Gamma/\Sigma}(H)[-1])$ is
\end{comment} 
\end{proof}
In \cite[Section 5]{MZ:OGradytype}, the authors give a classification of walls for stability for objects with non-primitive Mukai vector with square $2$ and divisibility $2$ on a K3 surface. In our more general noncommutative setting, we expect similar results hold for the singular moduli space $M_\sigma(2\lambda_1+2\lambda_2)$. %In particular, using that generically the polarization class of $M_\sigma(v)$ is $\lambda_1-\lambda_2$ and performing some elementary computations, it is possible to see that when the cubic fourfold $Y$ is in the divisor $\CC_{12}$, the stability condition $\sigma$ is on a flopping wall in the component of $\Stab^\dag(\Ku(Y))$. 
\begin{question}
Let $Y$ be in $\CC_{12}$, $\Gamma$, $\Sigma$ and $X$ be as those in the lemma. We expect that $\II_{\Sigma/ X}(H)$ is always $\sigma$-stable. Moreover, the object $E_\Gamma$ is strictly $\sigma$-semistable and $\sigma$ is on  the flopping wall predicted by \cite{MZ:OGradytype}.
\end{question}

\section{Application: Lagrangian fibration and twisted family of intermediate Jacobians}\label{sec:Lagrangian}

We are now ready to prove Theorem \ref{thm_compacttwistedJac}. Let $Y$ be a smooth cubic fourfold, and fix
\[v_0=\lambda_1+\lambda_2, \, v=2\lambda_1+2\lambda_2.\]
By \cite{BLM+}, we have stability conditions on $\Ku(Y)$ with full support property. In particular, we choose $\sigma_0$ which is generic with respect to $v$, and also is in a chamber whose closure contains the stability condition $\sigma$. By Theorem \ref{thm_OG10}, there exists a projective moduli space $M:=M_{\sigma_0}(v)$, which admits a projective hyperk\"ahler resolution $\tM$, deformation equivalent to O'G10.

For a very general $Y$, we can just take $\sigma_0=\sigma$. However, the example in Section \ref{sec:C12} shows that a change of the stability condition is necessary in special cases.

Recall from Section \ref{sec:stabobjinM} that, for every elliptic quintic $\Gamma$ contained in a \emph{smooth} hyperplane section of $Y$, we have an object $E_\Gamma \in \Ku(Y)$. By Theorem \ref{thm:EgammaECinkustab} and our choice of $\sigma_0$, we know that $E_\Gamma$ is $\sigma_0$-stable. We further denote by $M_0$ the locus of the objects of the form $E_\Gamma$ in $M$, which is identified with an open subvariety $i:M_0 \to \tM$. Similarly, by Theorem \ref{thm:EgammaECinkustab} we have a $\sigma_0$-stable object $E_C$ for every smooth conic $C$ in $Y$.

Recall that by Proposition \ref{prop:EGamaECinKu}(2), each $E_\Gamma$ is supported on a smooth cubic threefold. We define a liner series on $M_0$: for every point $x\in Y$, consider the divisors
\[D_x:= \{E_\Gamma\;|\;\Hom(E_\Gamma, \OO_x)\neq 0\}.\]
Using the embedding $Y \to \P^5$, these span a linear series on $M_0$, denoted by $|D|$. This linear series induces a morphism
\[\pi_0: M_0 \to \P_0 \subset \P^{5\vee},\]
which sends each $E_\Gamma$ to its support. Here $\P_0$ parameterizes smooth hyperplane sections of $Y$. Note that the fiber of $\pi_0$ is the moduli of $E_\Gamma$ on a fixed smooth cubic threefold, hence is affine and irreducible by Remark \ref{rem:instanton} and Proposition \ref{prop:EGamaECinKu}(2). Hence the elements in $|D|$ are prime divisors.

Now define a line bundle on $\tM$ as follows: by taking closure, each divisor in $|D|$ extends to a prime divisor in $\tM$. The closures of generic elements in $|D|$ remain linear equivalent. This defines a line bundle $\LL$. We use $|\LL|$ to denote the complete linear series associated to $\LL$, which is at least 5-dimensional by our construction.

Recall that two birational hyperk\"ahler manifolds are isomorphic outside a locus with codimension at least $2$ (see \cite[Section 2.2]{Huy:birsymplec}), hence the linear bundles on each are naturally identified. Now we have the following result by Matsushita:

\begin{prop}\label{birational_model}
There exists a projective hyperk\"ahler manifold $N$ birational to $\tM$, with the following properties:
\begin{enumerate}
    \item[\rm{a)}] the birational map restricts to an isomorphism away from $\emph{Bs}(\LL)$;
    \item[\rm{b)}] the induced line bundle $\LL'$ on $N$ is nef.
\end{enumerate}
\end{prop}

\begin{proof}
Note that $|\LL|$ on $\tM$ has no fixed divisor (fixed component), as it contains prime divisors given as closures of elements in $|D|$. Now the existence of $N$ with a) and b) follows from \cite[Prop 1]{Matsushita:BeauvilleConj}.
\end{proof}

The aim of this section is to prove the following theorem:

\begin{thm}
\label{thm_L'semiample}
The line bundle $\LL'$ on $N$ is semiample. A multiple of it induces a Lagrangian fibration $\pi: N \to B$.
\end{thm}

\begin{rmk}
We do not know whether $B\cong \P^5$, though by construction $B$ contains the open subset $\P_0$. It is in general a conjecture that the base of a Lagrangian fibration on a hyperk\"ahler manifold is always isomorphic to a projective space.
\end{rmk}

To prove this theorem we need to introduce one more construction. Denote by $\XX \to \P_0$ the family of smooth hyperplane sections of $Y$. In \cite{Voisin:twisted}, the twisted family of intermediate Jacobians of $p: J \to \P_0$ was constructed, where the fiber $J_t$ is the twisted intermediate Jacobian of the cubic threefold $X_t$ for each $t\in \P_0$. Note that the relative Fano variety of lines naturally embeds into $J$, and we denote the image by $F$. The subvariety relevant to our case is the image of $F$ under the relative involution on $J$, denoted by $-F$.

Now for the family $\XX \to \P_0$, consider the relative moduli space $\tJ \to \P_0$ of semistable instanton sheaves. By Remark \ref{rem:instanton}, each fiber $\tJ_t$ is isomorphic to the blowup of $J_t$ along the involution of the Fano surface. We have the following relationship of $\tJ$ and $J$.

\begin{prop}\label{prop:blowup_J}
The space $\tJ$ is isomorphic to the blowup of $J$ along the involution of the relative Fano surface $-F$.
\end{prop}

\begin{proof}
Note that there exists a quasi-universal family on $\tJ$ of instanton sheaves with second Chern classes given by 1-cycles of degree 2. By \cite[Theorem 4.8]{Druel:Instanton}, there exists a morphism $\tJ \to J$ induced by taking the second Chern class. By the previous discussion, we know this morphism is birational, with exceptional divisor in $\tJ$ mapped to $-F\subset J$. Now the result follows from the universal property of blowup.
\end{proof}

Now we have the following observation.

\begin{lem}
The variety $J_0 := J-(-F)$ is isomorphic to an open subset of $\tM$. More precisely, it is isomorphic to the union of $M_0$ and the open subset of the exceptional divisor over the locus parametrizing objects of the form $P_{\ell_1}\oplus P_{\ell_2}$ for disjoint lines. Moreover, this open set is disjoint from the base locus of $|\LL|$ on $\tM$.
\end{lem}

\begin{proof}
Proposition \ref{prop:blowup_J} implies that $J_0$ can be identified with the moduli space parametrizing instanton sheaves $E_\Gamma$ and $\II_{\ell_1/X}\oplus \II_{\ell_2/X}$ for \textit{disjoint} lines on any smooth cubic threefold $X\subset Y$. Recall that $E_\Gamma$, viewed as a torsion sheaf on $Y$, is a $\sigma_0$-stable object in $\Ku(Y)$. By Proposition \ref{prop:EGamaECinKu}(3), the sheaf $\II_{\ell_1/X}\oplus \II_{\ell_2/X}$ projects to the $\sigma_0$-semistable object $P_{\ell_1}\oplus P_{\ell_2}$.

Hence the projection functor $\mathsf{pr}\iota_*$ induces a morphism $J_0 \to M$, which is an isomorphism over $M_0$. For the object $P_{\ell_1}\oplus P_{\ell_2} \in M$ with disjoint lines, the fiber of the morphism is an open set of the $\P^1$ parametrizing smooth cubic threefolds containing $\ell_1$ and $\ell_2$. Recall that $\tM$ is given by the blowup of $M$ along the singular locus. At $P_{\ell_1}\oplus P_{\ell_2}$, it is locally an $A_1$-singularity, and the resolution produces a $\P^1$-fiber. Now the result follows from the universal property of blowup.

The last assertion follows from the construction of $|\LL|$: it generically consists of the closure of $D_x \subset M_0$ in $\tM$. A point in the exceptional divisor of $\tM$ in the fiber over the point $P_{\ell_1}\oplus P_{\ell_2}$, for disjoint lines, is identified with a point in $J_0$, hence is associated to a cubic threefold $X$. Now choose $x \notin X$, then the point is not contained in the closure of $D_x$. This proves the statement.
\end{proof}

This implies the following result.

\begin{lem}
The line bundle $\LL'$ on $N$ is not big.
\end{lem}

\begin{proof}
As $J_0 \subset \tM$ is away from $\text{Bs}(\LL)$, by Proposition \ref{birational_model} a), we can identify $J_0\subset N$. The important observation is that $p: J \to \P_0$ is a projective morphism and $-F$ is of relative codimension three. Hence the open set $J_0$ can be covered by \textit{proper} curves that are contracted by $p$.

Now recall from \cite[Corollary 2.2.7]{Laz:Positivity1} that a divisor $D$ is big if and only if
\[nD=A+E,\]
for some positive integer $n$ such that $A$ is an ample divisor, and $E$ is an effective divisor. In our case, for any effective divisor $E$, we can always choose a $p$-exceptional proper curve $C\subset J_0\subset N$ not contained in $E$. With this choice we have
\[\LL'.C=0,\, A.C>0,\, E.C>0,\]
so $\LL'$ is not big.
\end{proof}

Now we are ready to prove Theorem \ref{thm_L'semiample}.

\begin{proof}[Proof of Theorem \ref{thm_L'semiample}]
The theorem follows from several results in the literature. Note that by construction $|\LL'|$ is at least 5-dimensional, hence the Iitaka dimension $\kappa(\LL')\geq 5$.
On the other hand, since $\LL'$ is nef but not big on $N$, by \cite[Cor 3.2]{Matrushita_Zhang}, we have $q(\LL')=0$, where $q$ is the Beauville-Bogomolov form. By \cite[Prop 24.1]{GrossHuybrechtsJoyce}, this implies that the numerical dimension $\nu(\LL')=5$. Hence we have
\[\kappa(\LL')=\nu(\LL').\]
By \cite[Theorem 6.1]{Kawamata:Pluricanonical}, $\LL'$ is semiample. Now the assertion follows from \cite[Thm 1]{Matsushita:Addendum}.
\end{proof}

The following result completes the proof of Theorem \ref{thm_compacttwistedJac}.

\begin{prop}
\label{prop_compactification}
The hyperk\"ahler manifold $N$ provides a compactification of $J$, i.e.
$$J \cong \pi^{-1}(\P_0) \subset N.$$
\end{prop}

\begin{proof}
We know that $J$, $\tJ$, $M$, $\tM$ and $N$ are birational to each other. Note that both $\pi: \pi^{-1}(\P_0) \to \P_0$ and $p: J \to \P_0$ are projective morphisms. Now since both $N$ and $J$ have symplectic structures, they are both relative minimal models. By \cite[Theorem 3.52]{KollarMori}, $\pi^{-1}(\P_0)$ and $J$ are isomorphic in codimension 1, hence related by a relative flop. Moreover, the exceptional loci are covers by rational curves contracted by $\pi$ and $p$. However, as $J \to \P_0$ is a family of abelian varieties, such a relative flop cannot exist. Hence, we know that $J \cong \pi^{-1}(\P_0)$.
\end{proof}

Note that this provides a modular construction of the results of \cite{Voisin:twisted} and \cite[Remark 1.10]{Sacca:birgeomJac} on the existence of a hyperk\"ahler compactification of $J$.

It remains an interesting question to determine all birational models of $N$ for very general $Y$, similarly to the work \cite{Sacca:birgeomJac}. We plan to study this in future work. In this paper, we focus on one flop between $N$ and $\tM$, which can be explicitly described by our construction. 

\begin{ex}
\label{ex_flopconics}
Recall that an open subset of the exceptional divisor of $\tJ$ parametrizes the sheaves of the form $F_C$ where $C\subset X$ is a conic. The blowdown morphism to $-F\subset J\subset N$ is defined by taking the residual line of $C$ in $X$. Hence the fiber of this morphism is isomorphic to the $\P^2$ parametrizing all conics contained in a fixed $X$ and residual to a fixed line.

On the other hand, the projection of $F_C$ into $\Ku(Y)$ gives the object $E_C$, which is $\sigma_0$-stable by Theorem \ref{thm:EgammaECinkustab}, and defines a point in $\tM$. Hence the fiber of this projection is isomorphic to the $\P^2$ parametrizing all cubic threefolds containing a fixed conic $C$. This explicitly describes a flop between $\tM$ and $N$.
\end{ex}

\begin{rmk}\label{rmk:compare_with_Voisin}
For a very general cubic fourfold $Y$, it is easy to see that the Picard rank of $\tM$ and $N$ is two. In this case, we know that their movable cones are identified, with boundaries given by the blow up and the Lagrangian fibration. This implies that for such $Y$, there exists a unique hyperk\"ahler compactification of the twisted family of intermediate Jacobians with a Lagrangian fibration structure. In particular, $\tM$ and $N$ are not isomorphic and $N$ is isomorphic to Voisin's construction in \cite{Voisin:twisted}.
\end{rmk}

\section{Application: elliptic quintics and MRC quotients} \label{sec_conjCastravet}
In this section we prove Proposition \ref{prop_objsinmodulispaceintro}. Let $Y$ be a smooth cubic fourfold, recall that we can write the semiorthogonal decomposition
$$\Db(Y)=\langle \OO_Y(-2H),\OO_Y(-H), \Ku(Y), \OO_Y \rangle.$$
%as the Serre functor of $\Db(Y)$ is $\mathsf S_Y(-)= (-) \otimes \OO_Y(-3H)[4]$.

Let $\Gamma \subset Y$ be an elliptic quintic, whose ideal sheaf is denoted by $\II_{\Gamma/Y}$. Recall from Definition \ref{def:EGammaEC} that we have the following projection $E_\Gamma$ in $\Ku(Y)$:
$$E_\Gamma:=\mathsf{pr}(\II_{\Gamma/Y}(H))=\mathsf{R}_{\OO_Y(-H)}\mathsf{R}_{\OO_Y(-2H)}\mathsf{L}_{\OO_Y} \II_{\Gamma/Y}(H) \in \Ku(Y).$$
By Proposition \ref{prop:EGamaECinKu}, if $\Gamma$ is non-degenerate and spanning a smooth cubic threefold $X \subset Y$, then $E_\Gamma\cong \iota_*F_\Gamma$ as that defined in \eqref{eq:defofFGamma}. In particular, it sits in the following short exact sequence in $\Coh(Y)$: 
\[ 0\to\OO_X(-H)\to E_\Gamma\rightarrow \II_{\Gamma/X}(H)\to 0.\]
Moreover, by Theorem \ref{thm:EgammaECinkustab}, the object $E_\Gamma$ is stable in the moduli $M=M_{\sigma_0}(2\lambda_1+2\lambda_2)$, where $\sigma_0$ is as chosen in Section \ref{sec:Lagrangian}.

Let $\CC$ be the connected component of the Hilbert scheme $\mbox{Hilb}^{5m}(Y)$ parametrizing elliptic quintics in $Y$.

\begin{prop}
\label{prop_objsinmodulispace}
There is a rational map $\rho: \CC \dashrightarrow M$ defined by the projection of $\II_{\Gamma/Y}(H)$ in $\Ku(Y)$, which is the maximally rationally connected (MRC) fibration of $\CC$.
\end{prop}
\begin{proof}
Consider the open subset $U \subset \CC$ parametrizing non-degenerate elliptic quintic curves on $Y$. %\todo{our definition of elliptic quintic in Section 4 contains the requirement of loc.comp.int and h0=1.} locally complete intersection elliptic quintics spanning a smooth cubic threefold, with $h^0(\OO_\Gamma)=1$.\todo{$\OO_\Gamma$?}
Let $\II$ be the universal family on $Y \times U$ parametrizing the objects $\II_{\Gamma/Y}(H)$. Then the projection $\FF$ of $\II$ in $\Ku(Y \times U)$ is a flat family of $\sigma_0$-stable objects in $\Ku(Y)$. Thus there is an induced morphism from $U$ to $M$, defining the rational map $\rho$ in the statement.

Recall that $E_\Gamma$ is an instanton bundle over its support $X$, and $\Gamma$ can be identified with the vanishing locus of a section of $E_\Gamma(H)$. For a generic section, the vanishing locus is a locally complete intersection, and is connected and reduced. So it satisfies our conditions on $\Gamma$. Hence we know that the general fibers of $\rho$ are rational. 

To see that $\rho$ is the MRC fibration of $\CC$, it is enough to note that by Proposition \ref{birational_model} there exists a hyperk\"ahler compactification of the locus $M_0$ parametrizing $E_\Gamma$. Now \cite[Lemma 1.4]{deJong_starr} proves the claim.
\end{proof}

A closely related question is about rational quartics on cubic fourfolds. The following was conjectured by Castravet \cite[Page 416]{deJong_starr}, and follows from our results in Section \ref{sec:Lagrangian}.

\begin{prop}\label{prop:quartics}
For any smooth cubic fourfold $Y$, the MRC quotient of the main component of the Hilbert scheme of rational quartics on $Y$ is (birational to) the twisted family of intermediate Jacobians $J$ of $Y$.
\end{prop}

\begin{proof}
It was observed in \cite{deJong_starr} that it is enough to show that $J$ is not uniruled. This follows from the existence of the hyperk\"ahler compactification of the twisted family in Proposition \ref{prop_compactification}.
\end{proof}

This is the only remaining case of MRC quotients of rational curves on cubic fourfolds: the case of degree $d \leq 3$ is classical, while $d \geq 5$ was treated in \cite{deJong_starr}.

\begin{rmk}
\label{rmk_conjCastravet}
Here we briefly recall the connection between elliptic quintics and rational quartics. It was proved in \cite[Section 8]{HarrisRothStarr} that for a generic elliptic quintic in a generic cubic threefold, we can choose a generic cubic scroll surface containing the curve, such that the residual curve is a smooth rational quartic. Along this line, it should be possible to show that the main components of the Hilbert schemes corresponding to these two cases are stably birational. We do not need this result and leave it as an open question.
\end{rmk}

\bibliography{all}                      % .bib-Datei
\bibliographystyle{halpha}     % .bst-Datei
\end{document}